\documentclass{article}   
   
\usepackage{latexsym,amsfonts,amsmath,epsfig,tabularx,amsthm,amssymb,enumitem,bm,mathrsfs}

\usepackage{microtype}
\usepackage[a4paper]{geometry}

\usepackage[draft=false,colorlinks,bookmarksnumbered,linkcolor=black, citecolor=black]{hyperref}

\usepackage{aliascnt}

\theoremstyle{plain}
\newtheorem{theorem}{Theorem}[section]

\newaliascnt{lem}{theorem}
\newtheorem{lemma}[lem]{Lemma}

\aliascntresetthe{lem}

\newaliascnt{prop}{theorem}
\newtheorem{prop}[prop]{Proposition}

\aliascntresetthe{prop}

\newaliascnt{cor}{theorem}
\newtheorem{corollary}[cor]{Corollary}

\aliascntresetthe{cor}

\newaliascnt{con}{theorem}

\aliascntresetthe{con}

\theoremstyle{definition}
\newaliascnt{defi}{theorem}
\newtheorem{defix}[defi]{Definition}

\aliascntresetthe{defi}
\newenvironment{definition}
  {\pushQED{\qed}\defix}
  {\popQED\enddefix}

\newaliascnt{ex}{theorem}

\aliascntresetthe{ex}

\theoremstyle{remark}
\newaliascnt{rem}{theorem}
\newtheorem{remarkx}[rem]{Remark}

\aliascntresetthe{rem}
\newenvironment{remark}
  {\pushQED{\qed}\remarkx}
  {\popQED\endremarkx}

\renewcommand{\d}{\,\mathrm{d}}																						
\renewcommand*{\epsilon}{\varepsilon}                                   
\renewcommand*{\rho}{\varrho}                                   
\newcommand*{\nach}{\rightarrow}                                        
\newcommand*{\sep}{\; \vrule \;}                                                
\newcommand*{\N}{\mathbb{N}}                                            
\newcommand*{\R}{\mathbb{R}}                                            
\renewcommand{\i}{\mathrm{i}}
\newcommand*{\Z}{\mathbb{Z}}                                            %
\newcommand*{\B}{\mathcal{B}}                                           
\newcommand*{\id}{\mathrm{id}}                                          

\DeclareMathOperator{\rank}{rank}					

\newcommand*{\norm}[1]{\left\| #1 \right\|}                             
\newcommand*{\abs}[1]{\left| #1 \right|}                                
\newcommand*{\floor}[1]{\left\lfloor #1 \right\rfloor}                  
\newcommand*{\ceil}[1]{\left\lceil #1 \right\rceil}                     
\newcommand*{\link}[1]{(\ref{#1})}                                      

\renewcommand{\tilde}[1]{ \widetilde{#1} }        											

\newcommand{\BIGOP}[1]{\mathop  
 {\mathchoice 
        {\raise-0.22em\hbox{\huge $#1$}} 
        {\raise-0.05em\hbox{\Large $#1$}}{\hbox{\large $#1$}}{#1}
 }}
\newcommand{\bigtimes}{\BIGOP{\times}}                                  

\def\e{\varepsilon}
\def\la{\lambda}
\def\lall{\Lambda^{{\rm all}}}
\def\lstd{\Lambda^{{\rm std}}}

\usepackage{float}
\usepackage{tikz}
\usetikzlibrary{matrix,arrows}

\usepackage[margin=0.4cm,format=hang]{caption}

\begin{document}   
   
\title{Notes on $(s,t)$-weak tractability:\\ A refined classification of problems with (sub)exponential information complexity}     
 
\author{ 
Pawe{\l} Siedlecki\thanks{University of Warsaw, Faculty of Mathematics, Informatics and Mechanics, ul. Banacha 2, 02-097 Warszawa, Poland. Email: psiedlecki@mimuw.edu.pl. 
This author has been supported by National Science Centre of Poland (DEC-2012/07/N/ST1/03200). This author gratefully acknowledges the support of Institute for Computational and Experimental Research in 
Mathematics (ICERM).}
\and
Markus Weimar\thanks{Corresponding author. Philipps-University Marburg, Faculty of Mathematics and Computer Science, Hans-Meerwein-Stra{\ss}e, Lahnberge, 35032 Marburg, Germany. Email: weimar@mathematik.uni-marburg.de. This author has been supported by Deutsche Forschungsgemeinschaft DFG (DA 360/19-1).}
}  
   
\maketitle

\begin{abstract}
\noindent
In the last 20 years a whole hierarchy of notions of tractability was proposed and analyzed by several authors. These notions are used to classify the computational
hardness of continuous numerical problems $S=(S_d)_{d\in\N}$ in terms of the behavior of their information complexity $n(\epsilon,S_d)$ as a function of the accuracy $\epsilon$ and the dimension
$d$. 
By now a lot of effort was spend on either proving quantitative positive results (such as, e.g., the concrete dependence on $\epsilon$ and $d$ within the well-established framework of polynomial tractability), or on qualitative negative results (which, e.g., state that a given problem suffers from the so-called curse of dimensionality).
Although several weaker types of tractability were introduced recently, the theory
of information-based complexity still lacks a notion which allows to quantify the exact (sub-/super-)exponential dependence of $n(\epsilon,S_d)$ on both parameters $\epsilon$ and $d$.
In this paper we present the notion of $(s,t)$-weak tractability which attempts to fill this gap. 
Within this new framework the parameters $s$ and $t$ are used to
quantitatively refine the huge class of polynomially intractable problems. 
For linear, compact operators between Hilbert spaces we provide characterizations of $(s,t)$-weak
tractability w.r.t.\ the worst case setting in terms of singular values. 
In addition, our new notion is illustrated by classical examples which recently attracted some attention. 
In detail, we study approximation problems between periodic Sobolev spaces and integration problems for classes of smooth functions.


\end{abstract}

\section{Introduction}
Let $S=(S_d)_{d\in\N}$ denote a multivariate numerical problem, i.e., a sequence of \emph{solution operators}~$S_d$, where each of them maps problem elements $f$ from a subset of some normed (\emph{source}) space~$F_d$ onto its solution $S_d(f)$ in some other (\emph{target}) space $G_d$. 
In the following we refer to the parameter~$d$ as the \emph{dimension} of the problem instance $S_d$.
Typical examples cover approximation problems (where $S_d$ is an embedding operator between spaces of $d$-variate functions) or integration problems (where $S_d(f)$ is defined as the integral of $f$ over some $d$-dimensional domain).

We are interested in the computational hardness of $S$ with respect to given classes of algorithms. 
This can be modeled by the \emph{information complexity} $n(\epsilon,S_d)$ which is defined as the minimal number of information operations that are needed to solve the $d$-dimensional problem with accuracy $\epsilon>0$:
\begin{equation}\label{def:abs_crit}
	n^{\mathrm{abs}}(\epsilon,S_d)
	:= \min\!\left\{ n\in\N_0 \sep e(n,d) \leq \epsilon \right\}.
\end{equation}
Therein the quantity $e(n,d)$ is defined as the minimal error (measured w.r.t.\ a given \emph{setting}) that can be achieved among all algorithms (within the class under consideration) that use at most $n\in\N_0$ information operations (degrees of freedom) on the input $f$ to approximate the exact solution $S_d(f)$.
The \emph{initial error} of the $d$-dimensional problem instance $S_d$ is denoted by 
\begin{equation*}
	\epsilon_{d}^{\mathrm{init}}
	:=e(0,d), \quad d\in\N.
\end{equation*}
Besides the information complexity with respect to the \emph{absolute error criterion} as defined in \link{def:abs_crit} we also consider the respective quantity w.r.t.\ the \emph{normalized error criterion},
\begin{equation}\label{def:norm_crit}
	n^{\mathrm{norm}}(\epsilon,S_d)
	:=\min\!\left\{ n\in\N_0 \sep e(n,d)\leq \epsilon\cdot \epsilon_{d}^{\mathrm{init}} \right\},
\end{equation}
which measures how many pieces of information are needed to reduce the initial error by some factor $\epsilon\in(0,1]$.
Typical classes of algorithms under consideration are, e.g., methods based on arbitrary linear functionals (information in $\Lambda^{\mathrm{all}}$), or algorithms which are allowed to use function values ($\Lambda^{\mathrm{std}}$) only. 
Moreover, one may stick to linear methods only, allow or prohibit adaption and/or randomization.
Possible settings include the \emph{worst case}, \emph{average case}, \emph{probabilistic}, and the \emph{randomized setting}.
For concrete definitions, explicit complexity results, and further references, see, e.g., the monographs \cite{NW08,NW10,NW12,TWW88}, as well as the recent survey \cite{Wei14a}.

For the ease of presentation, in what follows, we mainly focus our attention on linear algorithms and their worst case errors among the unit ball $\B(F_d):=\{f\in F_d \sep \norm{f\sep F_d}\leq 1\}$ of our source space $F_d$. That is, we set
\begin{equation*}
	e(n,d)
	:=e^\mathrm{wor}(n,d;\Lambda) 
	:= \inf_{A_{n,d}} \sup_{f\in \B(F_d)} \norm{S_d(f)-A_{n,d}(f) \sep G_d},
\end{equation*}
where the infimum is taken w.r.t.\ to all (deterministic, linear, non-adaptive) algorithms that use~$n$ information operations from the class $\Lambda \in \{\Lambda^{\mathrm{all}}, \Lambda^{\mathrm{std}}\}$.

In the last 20 years a whole hierarchy of notions of tractability was proposed and analyzed by several authors in order to classify the behavior of the information complexity $n(\epsilon,S_d)$ as a function of the accuracy $\epsilon$ \emph{and} the dimension $d$.
By now a lot of effort was spend on either proving \emph{quantitative} positive results (such as, e.g., the concrete polynomial dependence on~$\epsilon$ and $d$ in the well-established framework of \emph{polynomial tractability}), or on \emph{qualitative} negative results (which, e.g., state that a given problem suffers from the \emph{curse of dimensionality}); again see \cite{NW08,NW10,NW12}.
Although several weaker types of tractability were introduced recently \cite{GneWoz11,PapPet14,Sie13}, to the best of our knowledge, the theory of information-based complexity still lacks a notion which allows to \emph{quantify} the exact (sub-/super-) exponential dependence of $n(\epsilon,S_d)$ on both parameters~$\epsilon$ and~$d$.
The aim of this paper is to fill this gap. 
To this end, we define and analyze the notion of \emph{$(s,t)$-weak tractability} which was coined very recently in \cite{Wei14}.

The material is organized as follows:
in \autoref{sect:def} we introduce our new category of $(s,t)$-weakly tractable problems and investigate relations with existing classes of tractability.
The subsequent \autoref{sect:hilbert} which deals with linear problems defined between Hilbert spaces then contains our main results. 
First of all, in \autoref{subsect:general}, we provide a characterization of $(s,t)$-weak tractability for general (linear and compact) Hilbert space problems $S=(S_d)_{d\in\N}$ w.r.t.\ the worst case setting (for both error criteria and information from $\Lambda^{\mathrm{all}}$) in terms of their singular values $\lambda^{(d)}=(\lambda_{d,j})_{j\in\N}$, $d\in\N$.
In \autoref{subsect:tensor}, we focus on the important subclass of tensor product problems between Hilbert spaces and prove corresponding assertions which rely on the univariate singular values $\lambda^{(1)}=(\lambda_{j})_{j\in\N}$ only.
To conclude this part of the paper, \autoref{subsect:standard} is devoted to a comparison of the power of function values (information from $\Lambda^{\mathrm{std}}$) and of general linear information ($\Lambda^{\mathrm{all}}$) for the specific problem of (weighted) multivariate approximation in the worst case setting as studied, e.g., in \cite[Chapter~26]{NW12}.
Finally, \autoref{sect:applications} deals with concrete examples recently studied by other authors. 
Here we investigate $(s,t)$-weak tractability for embeddings of periodic Sobolev spaces on the torus $\mathbb{T}^d$ for different norms (based on Fourier coefficients) in the source spaces. 
As a byproduct we close a gap in the characterization of (classical) weak tractability \cite[Theorem~5.5]{KueSicUll14}. 
Furthermore, we derive a positive complexity result for the integration problem of smooth functions based on error bounds recently published in \cite{HinNovUll14,HinNovUll+14}.

\section{Definition and simple properties}\label{sect:def}

Let $S=(S_d)_{d\in\N}$ denote a multivariate problem in the sense of the previous section and let $n^{\mathrm{crit}}(\epsilon,S_d)$, $\mathrm{crit}\in\{\mathrm{abs},\mathrm{norm}\}$, denote its information complexity with respect to the absolute or normalized error criterion in some fixed setting, respectively.
In order to quantify polynomial intractability, we generalize the by now classical notion of weak tractability as follows:

\begin{definition}\label{defi:st-WT}
If for some \emph{fixed} parameters $s,t \geq 0$ it holds
\begin{gather}\label{def:st-weak}
		\lim_{\epsilon^{-1}+d\nach\infty} \frac{\ln n^{\mathrm{crit}}(\epsilon,S_d)}{\epsilon^{-s}+d^{\,t}}=0,
\end{gather}
then the problem $S$ is called \emph{$(s,t)$-weakly tractable}.
\end{definition}
Roughly speaking, this means that we have $(s,t)$-weak tractability if the information complexity is neither exponential in $d^{\,t}$, nor in $\epsilon^{-s}$. Thus, varying $s$ and $t$ we are now able to quantify a (sub-/super-)exponential behavior of $n^{\mathrm{crit}}(\epsilon,S_d)$ in $\epsilon$ and/or in $d$.
As usual, the limit in \link{def:st-weak} is taken w.r.t.\ all two-dimensional sequences $((\epsilon_k,d_k))_{k\in\N}\subset(0,1]\times\N$ such that 
\begin{equation*}
	\epsilon_k < 
	\begin{cases}
		\epsilon_{d_k}^{\mathrm{init}} & \text{if} \quad \mathrm{crit} = \mathrm{abs},\\
		1 & \text{if} \quad \mathrm{crit} = \mathrm{norm},
	\end{cases}
\end{equation*}
for each $k\in\N$ and $\epsilon_k^{-1}+d_k\nach\infty$, as $k$ approaches infinity. 
Note that, without loss of generality, in what follows we always assume that $\epsilon_{d}^{\mathrm{init}}>0$ for every $d\in\N$ such that $n^{\mathrm{crit}}(\epsilon,S_d) \geq 1$ for all $\epsilon$ and $d$ under consideration.

In the subsequent \autoref{rem:parameter} we justify that the most interesting ranges of parameters are $0<s$ and $0<t\leq 1$.

\begin{remark}\label{rem:parameter}
Assume that the problem $S$ satisfies \link{def:st-weak} with $s=0 \leq t$. 
Given $d=d_0\in\N$ arbitrarily fixed we consider sequences $(\epsilon_k)_{k\in\N}$ for which $\min\{1,\epsilon_{d_0}^{\mathrm{init}}\} > \epsilon_1 \geq \epsilon_2 \geq \ldots > 0$.
Then \link{def:st-weak} and the fact that $n^{\mathrm{crit}}(\epsilon,S_d)$ is monotone in the accuracy $\epsilon$, i.e., $n^{\mathrm{crit}}(\epsilon_k,S_{d_0})\leq n^{\mathrm{crit}}(\epsilon_{k+1},S_{d_0})$ for all $k\in\N$, implies that $n^{\mathrm{crit}}(\epsilon_k,S_{d_0})=1$ for each $k$.
Hence, the problem is trivial since in every dimension it can be solved with arbitrary accuracy using only one piece of information on the input. In particular, such problems are strongly polynomially tractable.

Now suppose $S$ satisfies \link{def:st-weak} with $s>0=t$. Let $\epsilon=\epsilon_0\in(0,1)$ be arbitrarily fixed and consider the sequence $((\epsilon_0,d))_{d\in\N}$. 
Then, for the normalized error criterion, equation~\link{def:st-weak} implies that $n^{\mathrm{norm}}(\epsilon_0,S_d)\nach 1$, as $d\nach\infty$. Thus, for all dimensions $d$ larger than a certain $d^*(\epsilon_0)\in\N$ the problem $S$ can be solved to within the threshold $\epsilon_0$ using again only one information operation.
When dealing with the absolute error criterion every $d$ such that $\epsilon_{d}^{\mathrm{init}} \leq \epsilon_0$ satisfies $n^{\mathrm{abs}}(\epsilon_0,S_d)=0$. 
For the subsequence of all remaining $d$ the same argument as before applies. 
Therefore a problem for which $(s,0)$-weak tractability holds is trivial in the sense that asymptotically in $d$ it can be solved with arbitrary accuracy using at most one piece of information on the input.

Finally consider a problem $S$ which is $(s,t)$-weakly tractable with $s\geq 0$ and $t>1$ and, for simplicity, assume that there exists some positive constant $c<1$ such that $\epsilon_{d}^{\mathrm{init}}>c$ for all $d\in\N$.
Then the information complexity of $S$ is allowed to be lower bounded by $C \cdot (1+\gamma)^d$ for some $C,\gamma>0$, all $d\in\N$, and some fixed $\epsilon \in (0,c)$. Hence, $S$ may suffer from the curse of dimensionality.
\end{remark}

Next we compare the class of problems defined in \autoref{defi:st-WT} with existing notions of tractability. 
For this purpose recall that $S=(S_d)_{d\in\N}$ is called \emph{weakly tractable} if its information complexity is neither exponential in $\epsilon^{-1}$, nor in $d$, and $S$ is said to be \emph{uniformly weakly tractable} if $n^{\mathrm{crit}}(\epsilon,S_d)$ is not exponential in \emph{any positive power} of $\epsilon$ and/or $d$; see \cite{NW08} and \cite{Sie13}, respectively.
Hence, from \autoref{defi:st-WT} we immediately deduce the following proposition.
\begin{prop}
	Obviously,
	\begin{itemize}
		\item $S$ is uniformly weakly tractable if and only if it is $(s,t)$-weakly tractable for all $s,t>0$.
		\item $S$ is weakly tractable (in the classical sense) if and only if it is $(1,1)$-weakly tractable.
		\item for all $0 \leq s \leq \sigma$ and $0 \leq t \leq \tau$ fixed $(s,t)$-weak tractability implies $(\sigma,\tau)$-weak tractability.
	\end{itemize}
\end{prop}

Thus, in the tractability hierarchy, $(s,t)$-weak tractability with parameters $0<s,t<1$ is located in-between uniform weak tractability and classical weak tractability; see \autoref{fig} below.
Moreover, the examples in \autoref{sect:sobolev} below show that these inclusions are strict.

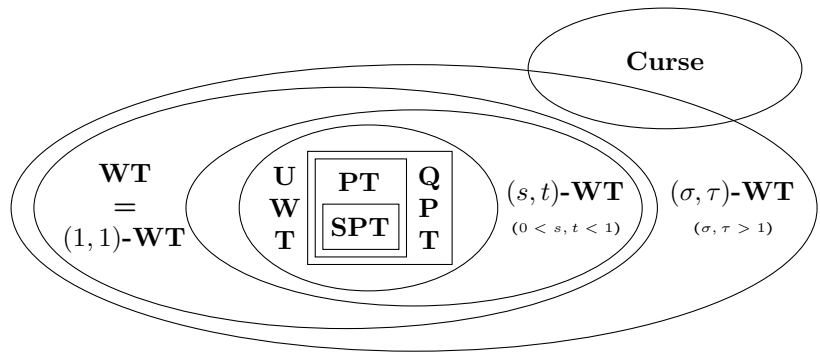
\begin{figure}[!h]
  \begin{center}
\renewcommand{\baselinestretch}{1.0}\normalsize
    \begin{tikzpicture}
    \draw (0,0) rectangle (1,0.6); 
    	\node at (0.6, 0.3) {\parbox{1cm}{\bf SPT}};
    \draw (-0.1,-0.1) rectangle (1.1,1.2); 
       	\node at (0.7, 0.9) {\parbox{1cm}{\bf PT}};
    \draw (-0.2,-0.2) rectangle (1.7,1.3); 
	    \node at (1.4, 0.55) {\parbox{1cm}{\centering\bf{Q\\ P \\ T}}};
    \draw (0.6,0.55) ellipse (1.7cm and 1.1cm); 
	    \node at (-0.5, 0.55) {\parbox{1cm}{\centering\bf{U \\ W\\ T}}};
	 \draw (1.2,0.55) ellipse (3cm and 1.3cm); 	 
		 \node at (3.2, 0.55) {\parbox{2cm}{\centering\bf{$(s,t)$-WT\\ \tiny{($0<s,t<1$)}}}};   
     \draw (0.3,0.55) ellipse (4.1cm and 1.6cm); 	 
	     \node at (-2.6, 0.55) {\parbox{2cm}{\centering \bf{WT\\ = \\$(1,1)$-WT}}};   
     \draw (1.2,0.55) ellipse (5.3cm and 1.9cm); 	
	     \node at (5.4, 0.55) {\parbox{2cm}{\centering\bf{$(\sigma,\tau)$-WT\\ \tiny{($\sigma,\tau>1$)}}}};   
     \draw (4.5,2.4) ellipse (1.8cm and 0.8cm); 
	      \node at (4.5, 2.5) {\parbox{2cm}{\centering \bf{Curse}}};   
    \end{tikzpicture}
    \renewcommand{\baselinestretch}{1.50}\normalsize
    \caption{\label{fig}Interrelation of $(s,t)$-weak tractability ($(s,t)$-WT) with strong polynomial/ polynomial/quasi-polynomial tractability (SPT/PT/QPT), as well as with uniform weak/weak tractability (UWT/WT), and the curse of dimensionality.}
  \end{center}
\end{figure}

\section{Compact linear problems defined between Hilbert spaces}\label{sect:hilbert}
\subsection{General Hilbert space problems}\label{subsect:general}
In this section we consider problems $S=(S_d)_{d\in\N}$ defined between arbitrary Hilbert spaces $H_d$ and $G_d$. That is, for every $d\in\N$ we assume that $S_d\colon H_d \nach G_d$ is a linear and compact operator which is characterized by its (squared) singular values $\lambda^{(d)}=(\lambda_{d,j})_{j\in\N}$; see, e.g., \cite{NW08} or \cite{Wei14a}. 
Without loss of generality we restrict our attention to those problems 
$S$ which are defined over infinite dimensional, separable Hilbert spaces. While 
this assumption simplifies our analysis, it does not harm the generality 
of our investigations; cf.\ \cite[Remark 2.6]{Wei14a} for details. 

Moreover, we may assume that for all $d$ the sequences $\lambda^{(d)}$ are not trivial and possess a non-increasing ordering:
\begin{equation*}
	\lambda_{d,1} \geq \lambda_{d,2}\geq \ldots \geq 0
	\quad \text{with} \quad \lambda_{d,1}>0, \qquad d\in\N.
\end{equation*}
Then none of the $S_d$'s is the zero operator and the initial error in dimension $d$ is given by $\epsilon_d^{\mathrm{init}} = \norm{S_d} = \lambda_{d,1}^{1/2}$.
We study the worst case setting with respect to the absolute and the normalized error criterion for the class $\Lambda^{\mathrm{all}}$ of all continuous linear functionals. To this end, given $d\in\N$, we define
\begin{equation}\label{def:CRI}
	\mathrm{CRI}_d 
	:= \begin{cases}
		1 & \text{if} \quad \mathrm{crit}=\mathrm{abs},\\
		\lambda_{d,1} & \text{if} \quad \mathrm{crit}=\mathrm{norm}.
	\end{cases}
\end{equation}
Recall that in this setting the $n$th optimal algorithm in dimension $d$ is given by the image (under~$S_d$) of the orthogonal projection of the input onto the subspace spanned by the eigenelements $\eta_{d,j}$ of the positive semi-definite, self-adjoint, and compact operator
\begin{equation*}
	W_d := S_d^* \circ S_d\colon H_d\nach H_d
\end{equation*}
which correspond to the $n$ largest eigenvalues $\lambda_{d,1},\ldots,\lambda_{d,n}$. Furthermore, from the general theory it follows that $e(n,d)=\sqrt{\lambda_{d,n+1}}$ and thus
\begin{equation}\label{eq:explicit_n}
	n^{\mathrm{crit}}(\epsilon,S_d) 
	= \min\!\left\{n\in\N_0 \sep \lambda_{d,n+1}\leq \epsilon^2\, \mathrm{CRI}_d \right\}, 
	\qquad \mathrm{crit}\in\{\mathrm{abs},\mathrm{norm}\}.
\end{equation}
For details and further reading we again refer to \cite{NW08} and \cite{Wei14a}. 

In the sequel we derive necessary and sufficient conditions for $(s,t)$-weak tractability of general Hilbert space problems $S=(S_d)_{d\in\N}$.
We start with conditions for the non-limiting case $\min\{s,t\}>0$. Afterwards, the analysis is completed by results for the cases in which $s=0$ and/or $t=0$.

\subsubsection{Non-limiting case}
\begin{theorem}\label{thm:general}
	Let $S=(S_d)_{d\in\N}$ be defined as above and consider the worst case setting w.r.t.\ the absolute or normalized error criterion for the class $\Lambda^{\mathrm{all}}$. In addition, let $s,t>0$ and $0<\overline{\beta}<1$ be fixed. 
	Then $S$ is $(s,t)$-weakly tractable if and only if the following two conditions hold:
	\begin{enumerate}[label=(C\arabic{*}), ref=C\arabic{*}]
		\item \label{cond:1} For every $d\in\N$ we have
		\begin{equation*}
			\lim_{j\nach\infty} \frac{\lambda_{d,j}}{\mathrm{CRI}_d} \, \ln^{2/s}\! j = 0.
		\end{equation*}
		\item \label{cond:2} There exists a function $f_{s,t}\colon (0,\overline{\beta}]\nach \N$ such that
		\begin{equation*}
			L_{s,t}
			:= \sup_{\beta\in (0,\overline{\beta}]} \frac{1}{\beta^{2/s}} \sup_{\substack{d\in\N,\\d\geq f_{s,t}(\beta)}} \sup_{\substack{j\in\N,\\j\geq\ceil{\exp(d^{\,t}\sqrt{\beta})}+1}} \frac{\lambda_{d,j}}{\mathrm{CRI}_d} \, \ln^{2/s}\! j < \infty.
		\end{equation*}
	\end{enumerate}
\end{theorem}
As usual, here and in what follows $\ln^{2/s}\! j$ means $\left[\ln(j)\right]^{2/s}$ and $\ceil{x}$ denotes the smallest natural number larger than or equal to $x>0$. Furthermore, $\floor{y}$ is the largest integer smaller than or equal to $y \geq 0$ and we use $(z)_+$ as a shorthand for the maximum of $z\in\R$ and zero.

\begin{remark}
Note that if we set $s:=t:=1$ and $\overline{\beta}:=1/2$, then \autoref{thm:general} coincides with the characterization of weak tractability stated in \cite[Theorem~5.3]{NW08}. 
Moreover, the proof given below shows that we have uniform weak tractability if and only if \link{cond:1} and \link{cond:2} hold for every $s=t>0$ and some fixed $0<\overline{\beta}<1$. 
Hence, we also generalized \cite[Theorem~4.1]{Xu14} in which $\overline{\beta}=1/2$.
\end{remark}

\begin{proof}[Proof of \autoref{thm:general}]
The proof can be derived using essentially the same arguments as exploited in \cite{NW08} and \cite{Xu14}, respectively. However, at some points additional estimates are needed as we shall now explain.
To this end, let $s,t>0$, as well as $0<\overline{\beta}<1$, be fixed arbitrarily.

\emph{Step 1}. We show that $(s,t)$-weak tractability implies \link{cond:1}. 
From \link{def:st-weak} we infer that for all $\beta\in(0,\overline{\beta}]$ there exists a natural number $M_{s,t}(\beta)$ such that for all pairs $(\epsilon,d)\in (0,1)\times \N$ with $\epsilon< \epsilon_d^{\mathrm{init}}$ and $\epsilon^{-1}+d\geq M_{s,t}(\beta)$ it holds
\begin{equation*}
	\frac{\ln n^{\mathrm{crit}}(\epsilon,S_d)}{\epsilon^{-s}+d^{\,t}}
	\leq \beta,
	\qquad \text{i.e.,} \qquad
	n^{\mathrm{crit}}(\epsilon,S_d) 
	\leq \floor{\exp\!\left(\beta\,(\epsilon^{-s}+d^{\,t})\right)}.
\end{equation*}
(W.l.o.g. we may assume that $M_{s,t}(\beta) > \min\{1,\epsilon_{d_0}^{\mathrm{init}}\}^{-1}+d_0$ for some fixed $d_0\in\N$.)
Using \link{eq:explicit_n} this yields that for all
\begin{equation}\label{eq:j}
	j \geq \floor{\exp\! \left(\beta\,(\epsilon^{-s}+d^{\,t})\right)} + 1
\end{equation}
it is $\lambda_{d,j} \leq \epsilon^2\, \mathrm{CRI}_d$.
In the case of equality in \link{eq:j} we clearly have $\epsilon^s \leq \beta / (\ln(j-1)-\beta\, d^{\,t})_+$. 
In conclusion, it holds
\begin{equation*}
	\frac{\lambda_{d,j}}{\mathrm{CRI}_d} 
	\leq \epsilon^2
	\leq \frac{\beta^{2/s}}{\left( \ln(j-1)-\beta \, d^{\,t} \right)_+^{2/s}}
	\qquad \text{for} \qquad 
	0<\beta\leq\overline{\beta}, 
	\qquad
	j = \floor{\exp\!\left( \beta\,(\epsilon^{-s}+d^{\,t}) \right)} + 1,
\end{equation*}
and all pairs $(\epsilon,d)\in (0, \min\{1,\epsilon_{d}^{\mathrm{init}}\}) \times \N$ for which $\epsilon^{-1}+d\geq M_{s,t}(\beta)$.
All such pairs obviously satisfy
\begin{equation*}
	\epsilon^{-s}+d^{\,t} \leq c_{s,t} \, \left( \epsilon^{-1}+d \right)^{\max\{s,t\}},
\end{equation*}
where the constant $c_{s,t}\geq 1$ does not depend on $\epsilon$ and $d$.
Now, for arbitrarily fixed $\beta$ and $d$, we let $\epsilon$ vary in the interval $(0, \min\{1,\epsilon_{d}^{\mathrm{init}}\})$. Then the left-hand side of the last inequality can be arbitrary large 
while the smallest value of the right-hand side is given by $c_{s,t} \, M_{s,t}(\beta)^{\max\{s,t\}}$.
This finally yields
\begin{equation}\label{ineq:bound_lambda}
	\frac{\lambda_{d,j}}{\mathrm{CRI}_d} 
	\leq \frac{\beta^{2/s}}{ \left( \ln(j-1)-\beta \, d^{\,t} \right)_+^{2/s}}
	\qquad \text{at least for all} \qquad
	j > \floor{ \exp\! \left(\beta\,c_{s,t} \, M_{s,t}(\beta) ^{\max\{s,t\}} \right)}.
\end{equation}

Using \link{ineq:bound_lambda} we deduce condition \link{cond:1} as follows:
For each natural number $j$ larger than some $j_0=j_0(\overline{\beta},d,t) \geq 3$ we have that $\ln(j-1)>2\, \overline{\beta}\,d^{\,t}$ which clearly implies
\begin{equation*}
	\ln(j-1)-\beta\,d^{\,t} 
	> \frac{1}{4} \ln j 
	> 0, 
	\qquad \text{i.e.,} \qquad
	\frac{1}{\left( \ln(j-1)-\beta \,d^{\,t} \right)_+^{2/s}} 
	< \frac{4^{2/s}}{\ln^{2/s}\! j}.
\end{equation*}
Consequently, for all $j > \max\left\{ j_0, \floor{\exp\!\left( \beta\,c_{s,t} \, M_{s,t}(\beta)^{\max\{s,t\}} \right) } \right\}$ it holds
\begin{equation*}
	\frac{\lambda_{d,j}}{\mathrm{CRI}_d} \, \ln^{2/s}\! j
	\leq (4\,\beta)^{2s}.
\end{equation*}
Since $\beta\in(0,\overline{\beta}]$ was arbitrary, this shows \link{cond:1}.

To verify also \link{cond:2}, let $f_{s,t} \colon (0,\overline{\beta}]\nach\N$ be defined by $f_{s,t}(\beta):=\ceil{\left( c_{s,t} \, M_{s,t}(\beta)^{\max\{s,t\}} \right)^{1/t} }$.
Then for all $0<\beta\leq\overline{\beta}<1$ and each $d\in\N$ with $d\geq f_{s,t}(\beta)$ we naturally have
\begin{equation}\label{ineq:dt}
	d^{\,t} \geq f_{s,t}(\beta)^t \geq c_{s,t} \, M_{s,t}(\beta)^{\max\{s,t\}}.
\end{equation}
Further note that for $j\geq \ceil{ \exp\! \left( d^{\,t} \, \sqrt{\beta} \right) }+1 \geq 3$ it is $\ln(j-1) \geq d^{\,t} \, \sqrt{\beta}$ which shows
\begin{equation}\label{ineq:lnj}
	\ln(j-1)-\beta \, d^{\,t} 
	\geq \left( 1-\sqrt{\overline{\beta}} \right) \, \ln(j-1) > 0.
\end{equation}
From \link{ineq:dt} and the fact that $\sqrt{\beta}\geq \beta$ we moreover conclude that every such $j$ is admissible for \link{ineq:bound_lambda}.
Combining \link{ineq:bound_lambda} with \link{ineq:lnj} now proves that
\begin{align*}
	\frac{1}{\beta^{2/s}} \, \frac{\lambda_{d,j}}{\mathrm{CRI}_d} \, \ln^{2/s}\! j
	\leq \frac{	\ln^{2/s}\! j}{ \left( \ln(j-1)-\beta \, d^{\,t} \right)^{2/s}}
	\leq \left( \frac{\ln j}{\ln (j-1)} \right)^{2/s} \frac{1}{\left( 1- \sqrt{\overline{\beta}} \right)^{2/s}}
	\leq \left( \frac{2}{1- \sqrt{\overline{\beta}}} \right)^{2/s}
\end{align*}
is uniformly bounded for all $\beta \in ( 0, \overline{\beta}]$, each $d\geq f_{s,t}(\beta)$, and all $j\geq \ceil{ \exp\! \left( d^{\,t} \, \sqrt{\beta} \right) }+1$.
In other words, we have shown that $L_{s,t}$ defined in \link{cond:2} is finite, as claimed.

\emph{Step 2.} 
We are left with proving the converse implication, i.e., that the conditions \link{cond:1} and \link{cond:2} together imply $(s,t)$-weak tractability of $S$.
For this purpose let $\beta\in(0,\overline{\beta}]$ be fixed. Then \link{cond:1} ensures the existence of some integer $J_{s,t}(\beta)>2$ such that for all $j\geq J_{s,t}(\beta)$ and each $d\in\{1,2,\ldots,f_{s,t}(\beta)-1\}$ it holds
\begin{equation*}
	\frac{\lambda_{d,j}}{\mathrm{CRI}_d} \, \ln^{2/s}\! j \leq \beta.
\end{equation*}
Since for $\epsilon>0$ the condition $\beta/\ln^{2/s}\! j \leq \epsilon^2$ is equivalent to $j\geq \ceil{ \exp\!\left( \beta^{s/2} \, \epsilon^{-s} \right)}$, we conclude
\begin{equation*}
	\lambda_{d,j} \leq \epsilon^2 \, \mathrm{CRI}_d
	\quad \text{for all} \quad
	d<f_{s,t}(\beta), \quad \epsilon>0, \quad \text{and} \quad j \geq \max\!\left\{ J_{s,t}(\beta), \, \ceil{ \exp\!\left( \beta^{s/2} \, \epsilon^{-s} \right)} \right\}.
\end{equation*}
Hence, $n^{\mathrm{crit}}(\epsilon,S_d) \leq \max\!\left\{ J_{s,t}(\beta), \, \exp\!\left( \beta^{s/2} \, \epsilon^{-s} \right) \right\}$ for these $d$ and $\epsilon$.

If $d\geq f_{s,t}(\beta)$, then condition \link{cond:2} yields that for all $j\geq \ceil{ \exp\!\left( d^{\,t} \sqrt{\beta} \right)}+1$
\begin{equation*}
	\frac{1}{\beta^{2/s}} \, \frac{\lambda_{d,j}}{\mathrm{CRI}_d} \, \ln^{2/s}\! j	
	\leq L_{s,t}.
\end{equation*}
Additionally, given $\epsilon>0$ we note that in this case $\beta^{2/s} \, L_{s,t} / \ln^{2/s}\! j\leq \epsilon^2$ holds if and only if $j\geq \ceil{ \exp\!\left( \beta \, L_{s,t}^{s/2} \, \epsilon^{-s} \right)}$. 
Hence, for all $d \geq f_{s,t}(\beta)$ and $\epsilon>0$,
we obtain
\begin{equation*}
	\lambda_{d,j} 
	\leq \epsilon^2 \, \mathrm{CRI}_d
	\qquad \text{if} \qquad
	j \geq \max\! \left\{ \ceil{ \exp\!\left( d^{\,t} \sqrt{\beta} \right)}+1, \, \ceil{ \exp\!\left( \beta \, L_{s,t}^{s/2} \, \epsilon^{-s} \right)} \right\}
\end{equation*}
and therefore $n^{\mathrm{crit}}(\epsilon,S_d) \leq \max\! \left\{ 2\,\exp\! \left( d^{\,t}\, \sqrt{\beta} \right), \,\exp\!\left( \beta \, L_{s,t}^{s/2} \, \epsilon^{-s} \right) \right\}$,
where we used the estimate $\ceil{x} \leq x +1 \leq 2\, x$ which holds for $x\geq 1$.

Combining both the estimates on the information complexity $n^{\mathrm{crit}}(\epsilon,S_d)$, we see that for all $d\in\N$, every $\epsilon \in (0,1)$, and all $\beta\in(0,\overline{\beta}]$ it is
\begin{equation}\label{ineq:max}
	\frac{\ln n^{\mathrm{crit}}(\epsilon,S_d)}{\epsilon^{-s}+d^{\,t}} 
	\leq \begin{cases}
		\displaystyle\max\! \left\{ \frac{\ln J_{s,t}(\beta)}{\epsilon^{-s}+d^{\,t}}, \, \beta^{s/2} \right\} 
			&\quad \text{if} \quad d< f_{s,t}(\beta), \\
		\displaystyle\max\! \left\{ \sqrt{\beta} + \frac{\ln 2}{\epsilon^{-s}+d^{\,t}}, \, \beta \, L_{s,t}^{s/2} \right\} 
			&\quad \text{if} \quad d\geq f_{s,t}(\beta).
	\end{cases}
\end{equation}
Note that there exists a constant $c'_{s,t}>0$ such that for all $\epsilon$ and $d$ we have
\begin{equation*}
	\epsilon^{-s}+d^{\,t} \geq c'_{s,t} \, \left( \epsilon^{-1}+d \right)^{\min\{s,t\}}.
\end{equation*}
Hence, for all $\delta>0$ we find $\beta=\beta(\delta)\in(0,\overline{\beta}]$ and $N_{s,t}(\delta)\in\N$ which ensure that both the maxima in~\link{ineq:max} are less than $\delta$ for all admissible pairs $(\epsilon,d)$ with $\epsilon^{-1}+d\geq N_{s,t}(\delta)$.
In other words, we have shown \link{def:st-weak} and thus $S$ is $(s,t)$-weakly tractable which completes the proof.
\end{proof}

\subsubsection{Limiting case}
We turn to the limiting case of $(s,t)$-weak tractability in which $s$ and/or $t$ equal zero. 
In these cases the situation is completely different compared to \autoref{thm:general}, as the following theorem shows:

\begin{theorem}\label{thm:general0}
Let $S=(S_d)_{d\in\N}$ be defined as above and consider the worst case setting w.r.t.\ the absolute or normalized error criterion for the class $\Lambda^{\mathrm{all}}$. Then
\begin{itemize}
	\item $S$ is $(0,t)$-weakly tractable with $t\geq 0$ if and only if
		\begin{equation}\label{eq:trivial_lambda}
			\lambda_{d,2}=\lambda_{d,3}=\ldots=0	\qquad \text{for all} \quad d\in\N.
		\end{equation}
	\item $S$ is $(s,0)$-weakly tractable with $s>0$ if and only if the following two conditions hold:
		\begin{enumerate}[label=(C\arabic{*}), ref=C\arabic{*}, start=3]
		\item\label{cond:3} $\displaystyle\lim_{d\to\infty} \frac{\lambda_{d,2}}{\mathrm{CRI}_d} = 0$,
		\item\label{cond:4} For all $0<\overline{\beta}<1$ there exists a function $g_{s}\colon (0,\overline{\beta}]\nach \N$ such that
		\begin{equation*}
			L_{s}
			:= \sup_{\beta\in (0,\overline{\beta}]} \frac{1}{\beta^{2/s}} \sup_{d\in\N} \sup_{\substack{j\in\N,\\j\geq g_s(\beta)}} \frac{\lambda_{d,j}}{\mathrm{CRI}_d} \, \ln^{2/s}\! j < \infty.
		\end{equation*}
	\end{enumerate}		
\end{itemize}
\end{theorem}

\begin{remark}
Observe that \link{cond:4} actually implies that the convergence described by \link{cond:1} takes place uniformly in $d$.
\end{remark}

\begin{proof}[Proof of \autoref{thm:general0}]
\emph{Step 1.}
If $0=s\leq t$, then formula \link{def:st-weak} yields that $n^{\mathrm{crit}}(\epsilon,S_d)=1$ for each $d\in\N$ and every $\epsilon \in (0,\min\{1,\e_d^{\mathrm{init}}\})$; see \autoref{rem:parameter}. 
From \link{eq:explicit_n} we thus conclude the necessity of \link{eq:trivial_lambda}, but of course this condition is also sufficient for $(0,t)$-weak tractability with $t\geq 0$.

\emph{Step 2.}
Let us turn to $(s,0)$-weak tractability of $S$ with $s>0$.
If $S$ is $(s,0)$-weakly tractable, then 
\autoref{rem:parameter} yields that $n^{\mathrm{crit}}(\e,S_d) \leq 1$ for all $\e>0$ and every $d\geq d^*(\e)$. 
Hence, formula \link{eq:explicit_n} implies
\begin{equation*}
	\frac{\lambda_{d,2}}{\mathrm{CRI}_d} \leq \e
\end{equation*}
for all these $\e$ and $d$. Letting $\e$ tend to zero thus proves the limit condition \link{cond:3}.
To show also \link{cond:4} we can proceed as in Step 1 of the proof of \autoref{thm:general} to derive \link{ineq:bound_lambda} with $t=0$ for each $d\in\N$ and all $0<\beta \leq \overline{\beta}$. 
Under the additional assumption that $j\geq \ceil{\exp\!\left( \sqrt{\beta} \right)}$ we moreover have \link{ineq:lnj} with $t=0$, so that uniformly in $d$ and $\beta$ it holds
\begin{equation*}
	\frac{1}{\beta^{2/s}} \, \frac{\lambda_{d,j}}{\mathrm{CRI}_d} \, \ln^{2/s}\! j
	\leq \left( \frac{2}{1- \sqrt{\overline{\beta}}} \right)^{2/s},
\end{equation*}
provided that
\begin{equation*}
	j \geq g_s(\beta):= \max\left\{ \floor{ \exp\! \left(\beta\,c_{s,0} \, M_{s,0}(\beta)^s \right)}+1, \ceil{\exp\!\left( \sqrt{\beta} \right)} \right\}.
\end{equation*}
This proves the necessity of \link{cond:4}.

It remains to show sufficiency of \link{cond:3} and \link{cond:4} for $(s,0)$-weak tractability of $S$.
For this purpose, we argue similar to Step 2 in the proof of \autoref{thm:general0}:
From \link{cond:4} it follows that 
\begin{equation*}
	\lambda_{d,j} 
	\leq L_s\, \frac{\beta^{2/s}}{\ln^{2/s}\! j} \, \mathrm{CRI}_d
	\qquad \text{if} \quad j\geq g_s(\beta)
\end{equation*}
and, given $\epsilon>0$, we again have $L_{s} \, \beta^{2/s} / \ln^{2/s}\! j \leq \epsilon^2$ if and only if $j\geq \ceil{ \exp\!\left( \beta \, L_{s}^{s/2} \, \epsilon^{-s} \right)}$. 
Therefore we conclude that if $j$ satisfies both requirements, then $\lambda_{d,j} \leq \e^2 \, \mathrm{CRI}_d$ and consequently
\begin{equation}\label{est:n_final}
	\frac{\ln n^{\mathrm{crit}}(\epsilon,S_d)}{\epsilon^{-s}+d^{0}} 
	\leq \frac{\ln \max\! \left\{ g_s(\beta)-1, \, \exp\!\left(\beta \, L_s^{s/2} \e^{-s}\right) \right\} }{\epsilon^{-s}+1}
	\leq \max\! \left\{ \frac{\ln g_s(\beta)}{\epsilon^{-s}+1}, \, \beta \, L_s^{s/2} \right\}
\end{equation}
for all $\e>0$, $d\in\N$, and $0<\beta\leq \overline{\beta}$.
Now let us take any admissible double sequence $((\epsilon_k,d_k))_{k\in\N}$ with $\epsilon_k^{-1}+d_k\nach\infty$, as $k$ approaches infinity. 
Given $\delta>0$ we can choose $\beta$ as well as $c>0$ small enough such that both entries of the maximum in \link{est:n_final} are smaller than $\delta$ for all pairs $(\e,d)=(\e_k,d_k)$ with $k\in\N$ and $\e_k \leq c$.
It might happen that there remains a subsequence with $\epsilon_{k_\ell} > c$ for all $\ell\in\N$. Then $\epsilon_{k_\ell}^{-1}+d_{k_\ell} \nach \infty$ does not imply $\epsilon_{k_\ell}^{-s}+1 \nach \infty$. However, in this case $d_{k_\ell}\to\infty$, as $\ell\to\infty$, so that we can use the condition \link{cond:3} to obtain
\begin{equation*}
	\ln n^{\mathrm{crit}}(\e_{k_\ell},S_{d_{k_\ell}}) 
	\leq \ln n^{\mathrm{crit}}(c,S_{d_{k_\ell}})
	\leq \ln 1 = 0,
\end{equation*}
provided that $\ell$ is sufficiently large.
In conclusion, for all $\delta>0$ we find $k_0(\delta)\in\N$ so that
\begin{equation*}
	\frac{\ln n^{\mathrm{crit}}(\epsilon_k,S_{d_k})}{\epsilon_k^{-s}+d_k^{0}}<\delta
	\qquad \text{for all} \quad k \geq k_0(\delta).
\end{equation*}
This shows $(s,0)$-weak tractability and hence completes the proof.
\end{proof}

\subsection{Linear tensor product problems}\label{subsect:tensor}
Linear tensor product problems are important special cases of general Hilbert space problems discussed in the previous \autoref{subsect:general}.
Here the sequence of problem instances $S=(S_d)_{d\in\N}$ is generated by taking $d$-fold tensor products of some (univariate) compact linear operator $S_1\colon H_1 \rightarrow G_1$ between Hilbert spaces $H_1$ and $G_1$. That is, we set
\begin{equation*}
	S_d:=\bigotimes_{\ell=1}^d S_1 \colon H_d \rightarrow G_d,
	\qquad \text{where} \quad 
	H_d:=\bigotimes_{\ell=1}^d H_1
	\quad \text{and} \quad 
	G_d:=\bigotimes_{\ell=1}^d G_1
	\quad \text{for every} \quad 
	d\in\mathbb{N}.
\end{equation*}
As before we restrict ourselves to problems 
for which $H_1$ and thus all source spaces $H_d$ are infinite dimensional and separable. 

By $\left((\la_j,\eta_j)\right)_{j\in\mathbb{N}}$ we denote the sequence of eigenpairs of $W_1:=S_1^{*}\circ S_1\colon H_1\rightarrow H_1$ arranged w.r.t.\ a non-increasing ordering of the eigenvalues $\lambda_j$.
Once more we obviously have $\lambda_j\geq 0$ and $\lim_{j\rightarrow\infty}\lambda_j=0$ since $W_1$ is a positive semi-definite, self-adjoint, and compact operator between Hilbert spaces. To avoid triviality we again assume that $\lambda_1>0$. For $d\ge1$ we set
$$
W_d:=S_d^* \circ S_d \colon H_d\rightarrow H_d.
$$
Then, due to the imposed tensor product structure of $S_d$ and $H_d$, the set of eigenpairs of $W_d$ is given by
\begin{equation*}
	\left\{ (\lambda_{d,j},\eta_{d,j}) = \left( \prod_{\ell=1}^d \lambda_{j_\ell}, \bigotimes_{\ell=1}^d \eta_{j_\ell} \right) \sep j=(j_1,\ldots,j_d)\in\mathbb{N}^d \right\}
\end{equation*}
In particular, we have $\lambda_{d,(1,\ldots,1)}=\lambda_1^d = \norm{S_1}^{2d}=\norm{S_d}^{2}>0$ and thus $\e_d^{\mathrm{init}}=\lambda_1^{d/2}$ when dealing with the worst case setting.
Moreover, it is well-known that
\begin{equation}\label{eq:complexity}
	n^{\mathrm{crit}}(\e,S_d)
	= \#\left\{ j\in\mathbb{N}^d \sep \lambda_{d,j} > \e^2\,{\rm CRI}_d \right\}, \qquad \mathrm{crit} \in \{\mathrm{abs}, \, \mathrm{norm}\},
\end{equation}
where again ${\rm CRI}_d:=1$ for the absolute, and
${\rm CRI}_d:=\lambda_{1}^d$ for the normalized error criterion; cf.\ \link{def:CRI}.

In what follows we significantly extend the characterization of weak tractability for linear tensor product problems (as it can be found, e.g., in \cite[Theorem 5.5]{NW08} and \cite{PapPet09}) to the case of $(s,t)$-weak tractability.
For this purpose, we first derive conditions which are necessary and sufficient for the limiting case $\min\{s,t\}=0$. Afterwards we give characterizations for the remaining non-limiting cases.

\subsubsection{Limiting case}
Our first theorem for linear tensor product problems in the above sense characterizes $(s,t)$-weak tractability with $s=0$.
\begin{theorem}\label{twierdzonko}
	Let $S=(S_d)_{d\in\N}$ be a linear tensor product problem and consider the worst case setting  for the class $\Lambda^{\mathrm{all}}$. 
	Then the following assertions are equivalent:
	\begin{itemize}
			\item For $t\geq 0$ the problem $S$ is $(0,t)$-weakly tractable w.r.t.\ the absolute error criterion.
			\item For $t\geq 0$ the problem $S$ is $(0,t)$-weakly tractable w.r.t.\ the normalized error criterion.
			\item $\lambda_2=\lambda_3=\ldots=0$.
			\item Every problem instance $S_d$ can be solved exactly using only one piece of information.
	\end{itemize}	
\end{theorem}
\begin{proof}
The proof can be derived easily from the corresponding result for general Hilbert space problems given in \autoref{thm:general0} and the product structure of the eigenvalues $\lambda_{d,j}$ of $W_d$.
\end{proof}

In order to present a characterization of $(s,t)$-weak tractability for the case $t=0$ two preliminary lemmata are needed:
\begin{lemma}\label{lemacik3}
	Let $S_1$ be defined as above. Then, for every $s>0$,
	\begin{equation*}
		\lim_{n\to\infty} \frac{\la_n}{\ln^{-2/s}\! n}=0
		\qquad \text{if and only if} \qquad 
		\lim_{\e\to 0} \frac{\ln n^{\mathrm{abs}}(\e,S_1)}{\e^{-s}}=0.
	\end{equation*}
\end{lemma}
\begin{proof}
This lemma follows immediately from the proofs of \cite[Lemma 1 and Lemma 2]{Sie13}. 
\end{proof}

While \autoref{lemacik3} relates the decay of the sequence $\lambda^{(1)}=(\lambda_j)_{j\in\N}$ to the growth behavior of the univariate information complexity $n^\mathrm{abs}(\e,S_1)$ as $\epsilon\to 0$, \autoref{lemacikgorny} below deals with an upper estimate for the multidimensional case. Its proof is based on combinatorial arguments similar to those used in \cite{PapPet09} and in the proof of \cite[Theorem~5.5]{NW08}, respectively.

\begin{lemma}\label{lemacikgorny}
Given a linear tensor product problem 
$S=(S_d)_{d\in\mathbb{N}}$ let $S'_1:=\frac{1}{\sqrt{\la_1}}S_1$. 
Then for all $d\in\N$ and $\epsilon\in(0,1]$ there holds
\begin{equation*}
	n^{\mathrm{abs}}(\e, S_d)
	\leq d!\cdot \prod_{\ell=1}^d n^{\mathrm{abs}}\!\left( \left( \frac{\e}{\lambda_1^{d/2}}\right)^{1/\ell},\, S'_1 \right).
\end{equation*}
\end{lemma}

\begin{proof}
Let $\epsilon$ and $d$ be fixed.
Obviously, the sequence of eigenvalues $(\la'_n)_{n\in\mathbb{N}}$ of the operator $W'_1:=(S'_1)^* \circ S'_1$ related to the modified (univariate) problem instance $S'_1$ satisfies 
\begin{equation}\label{eq:rescale}
	\lambda'_n 
	= \frac{\lambda_n}{\lambda_1} \in [0,1]
	\qquad \text{for every} \quad n\in\N.
\end{equation}
Therefore we can rewrite the information complexity \link{eq:complexity} of the original (multivariate) problem instance $S_d$ as
\begin{equation}\label{eq:n_modified}
	n^{\mathrm{abs}}(\e, S_d) 
	= \#\!\left\{j=(j_1,\ldots,j_d)\in\mathbb{N}^d \sep \la'_{j_1}\cdot\ldots\cdot\la'_{j_d} > \frac{\e^2}{\la_1^d} \right\}.
\end{equation}

Suppose that $(j_1,\ldots,j_d)\in\mathbb{N}^d$ satisfies 
\begin{equation}\label{est:admissible}
	\la'_{j_1} \cdot \ldots \cdot \la'_{j_d} > \frac{\e^2}{\la_1^d}
\end{equation}
and let $\sigma^{*}:\{1,\ldots,d\}\rightarrow\{1,\ldots,d\}$ denote a permutation such that 
$j_{\sigma^{*}(1)}\geq j_{\sigma^{*}(2)}\geq\ldots\geq j_{\sigma^{*}(d)}$.
If we set
\begin{equation*}
	j_{\max}(\ell):=j_{\sigma^{*}(\ell)}
\end{equation*} 
for $\ell=1,\ldots,d,$ then 
\begin{align*}
	(\la'_{j_{\max}(\ell)})^\ell
	\geq \la'_{j_{\max}(1)} \cdot \ldots \cdot \la'_{j_{\max}(\ell)} \cdot 1 \cdot \ldots \cdot 1 
	\geq \la'_{j_1} \cdot \ldots \cdot \la'_{j_d}.
\end{align*}
Hence, from \link{est:admissible} it follows
\begin{equation*}
	\la'_{j_{\max}(\ell)} > \left(\frac{\e^{2}}{\la_1^{d}}\right)^{1/\ell}
\end{equation*}
and thus, due to \link{eq:complexity} applied for $S'_1$,
\begin{equation*}
	j_{\max}(\ell)
	\leq n^{\mathrm{abs}}\!\left( \left(\e/\la_1^{d/2}\right)^{1/\ell}, \, S'_1 \right),
	\qquad \ell=1,\ldots,d.
\end{equation*}

Now let us define the sets
\begin{eqnarray*}
	A \!\!\!&:=&\!\!\! \bigtimes_{\ell=1}^d \left\{1,2,\ldots,n^{\mathrm{abs}}\!\left( \left(\e/\la_1^{d/2}\right)^{1/\ell}, \, S'_1 \right)\right\},\\
	B \!\!\!&:=&\!\!\! \left\{(l_1,\ldots,l_d)\in\mathbb{N}^d \sep (l_{\sigma(1)},\ldots,l_{\sigma(d)})\in A
\ \ \text{for some}\ \ \sigma\in\Sigma_d\right\},
\end{eqnarray*}
where $\Sigma_d$ denotes the set of all permutations on $\{1,\ldots,d\}$.
Note that then
$$
	\left\{ j=(j_1,\ldots,j_d)\in\mathbb{N}^d \sep \la'_{j_1}\cdot\ldots\cdot\la'_{j_d} > \frac{\e^2}{\la_1^d} \right\}\subset B.
$$
Indeed, for every $j$ from this set the rearranged multiindex $(j_{\sigma^{*}(1)},\ldots,j_{\sigma^{*}(d)})$ belongs to $A$, i.e., by definition it is $j\in B$.
In conclusion, the representation \link{eq:n_modified} yields
\begin{equation*}
	n^{\mathrm{abs}}(\e,S_d)
	\leq \# B
	\leq d!\cdot \# A 
	= d!\cdot \prod_{\ell=1}^d n^{\mathrm{abs}}\!\left( \left(\e/\lambda_1^{d/2}\right)^{1/\ell},\, S'_1 \right),
\end{equation*}
as claimed.
\end{proof}

Now the characterization of $(s,0)$-weak tractability reads as follows:
\begin{theorem}\label{thm:tensor_t0}
	Let $S=(S_d)_{d\in\N}$ be a linear tensor product problem and consider the worst case setting for the class $\Lambda^{\mathrm{all}}$. Moreover, let $s>0$. Then
	\begin{itemize}
		\item $S$ is $(s,0)$-weakly tractable w.r.t.\ the normalized error criterion if and only if
		\begin{equation*}
			\lambda_2 = \lambda_3 = \ldots = 0.
		\end{equation*}
		\item $S$ is $(s,0)$-weakly tractable w.r.t.\ the absolute error criterion if and only if one of the following conditions applies:
			\begin{itemize}
				\item[1.)] $\lambda_2 = \lambda_3 = \ldots = 0$, or
				\item[2.)] $\lambda_1<1$ \quad and \quad $\displaystyle\lim_{j\to\infty} \frac{\lambda_j}{\ln^{-2/s}\! j}=0$.
			\end{itemize}
	\end{itemize}
\end{theorem}
\begin{proof}
\emph{Step 1.}
From \autoref{thm:general0} it follows that $(s,0)$-weak tractability of arbitrary Hilbert space problems~$S$ (in the sense of \autoref{subsect:general}) is equivalent to the conditions \link{cond:3} and \link{cond:4}. Since now we deal with linear tensor product problems $S$ we moreover know that
\begin{equation*}
	\frac{\lambda_{d,2}}{\mathrm{CRI}_d} 
	= \begin{cases}
		\displaystyle\frac{\lambda_1^{d-1}\cdot\lambda_2}{\lambda_1^d} = \frac{\lambda_2}{\lambda_1} & \text{if} \quad \mathrm{crit}=\mathrm{norm}, \\
		\lambda_1^{d-1}\cdot\lambda_2 & \text{if} \quad \mathrm{crit}=\mathrm{abs}
	\end{cases}
\end{equation*}
for every $d\in\N$. Consequently, for the normalized error criterion \link{cond:3}, i.e., $\lim_{d\to\infty} \lambda_{d,2}/\mathrm{CRI}_d=0$, holds if and only if $\lambda_2=0$. For the absolute error criterion \link{cond:3} it is equivalent to $\lambda_1<1$ or $\lambda_2=0$. Since $\lambda_2=0$ clearly implies that all $\lambda_j$, $j\geq 2$, equal zero, it also yields condition \link{cond:4}. This proves the assertion for the normalized error criterion, as well as the first part for the absolute error criterion.

\emph{Step 2.}
It remains to show that for $1>\lambda_1 \geq \lambda_2>0$ the problem is $S$ is $(s,0)$-weakly tractable if and only if $\lim_{j\to\infty} \lambda_j \, \ln^{2/s}\! j=0$. Condition \link{cond:4} shows that this limit condition is necessary. Indeed, \link{cond:4} particularly yields that for $d=1$, some function $g_s$, and every (small) $\beta>0$ it holds
\begin{equation*}
 	\sup_{\substack{j\in\N,\\j\geq g_s(\beta)}} \lambda_{j} \, \ln^{2/s}\! j \leq L_s \, \beta^{2/s}.
\end{equation*}

To prove sufficiency we apply \autoref{lemacikgorny} and obtain
\begin{align*}
	\frac{\ln n^{\mathrm{abs}}(\e, S_d)}{\e^{-s}+d^{\,0}} 
	&\leq \frac{d\, \ln d}{\e^{-s}+1} + \frac{d\, \ln n^{\mathrm{abs}}\!\left(\e/\lambda_1^{d/2}, S'_1\right)}{\e^{-s}+1} 
	\qquad \text{for each} \quad d\in\N\quad \text{and all} \quad 0<\e < \lambda_1^{d/2},
\end{align*}
where we have used the monotonicity of $n^{\mathrm{abs}}$ w.r.t.\ its first argument.
Note that that $\lambda_1<1$ ensures $\e\leq 1$. For $(s,0)$-weak tractability it is enough to show that both fractions tend to zero for all double sequences $((\e_k,d_k))_{k\in\N}$ with $\e_k<\lambda_1^{d_k/2}$ and $\e_k^{-1}+d_k\to\infty$, as $k$ approaches infinity.
To this end, note that
\begin{equation*}
	\e_k<\lambda_1^{d_k/2} \quad \text{if and only if} \quad d_k < \frac{\ln \e_k^{-1}}{\ln \lambda_1^{-1/2}}
\end{equation*}
and that double sequences which fulfill this necessarily satisfy 
\begin{equation}\label{lim:ek}
	\e_k \to 0, \qquad \text{as} \quad k \to\infty.
\end{equation}
(If $\e_k^{-1}$ is upper bounded, so is $d_k$. This would contradict $\e_k^{-1}+d_k\to\infty$.)
Hence, for the first fraction we conclude
\begin{equation*}
	\frac{d_k \cdot \ln d_k}{\e_k^{-s}+1} 
	\leq c \, \frac{\ln \e_k^{-1} \cdot \ln \ln \e_k^{-1}}{\e_k^{-s}} 
	\to 0, \qquad \text{as} \quad k \to\infty.
\end{equation*}

To handle also the second fraction we distinguish two cases: At first let us consider all subsequences $((\e_{k_\ell},d_{k_\ell}))_{\ell\in\N}$ for which $\delta:=\e_{k_\ell}/\lambda_1^{d_{k_\ell}/2}$ tends to zero, as $\ell\to\infty$.
Then, for $k=k_\ell$, we can estimate
\begin{align}
	\frac{d_k\, \ln n^{\mathrm{abs}}\!\left(\e_k/\lambda_1^{d_k/2}, S'_1\right)}{\e_k^{-s}+1} 
	&\leq \frac{\left( \e_{k}/\lambda_1^{d_{k}/2} \right)^{-s} }{\e_k^{-s}} \cdot \frac{\ln n^{\mathrm{abs}}\!\left(\e_k/\lambda_1^{d_k/2}, S'_1\right)}{\left( \e_{k}/\lambda_1^{d_{k}/2} \right)^{-s}} \nonumber\\
	&= d_k \, \lambda_1^{s\,d_{k}/2} \cdot \frac{\ln n^{\mathrm{abs}}\left(\delta, S'_1\right)}{\delta^{-s}}. \label{eq:2nd_frac}
\end{align}
Therein the first factor is uniformly bounded because of $\lambda_1<1$ and $s>0$. 
Note that our assumption $\lim_{j\to\infty} \lambda_j \, \ln^{2/s}\! j=0$ likewise holds for the rescaled sequence $\lambda_j'$ as defined in \link{eq:rescale}. Using \autoref{lemacik3} this proves that \link{eq:2nd_frac} vanishes for $\delta\to 0$, i.e., if $\ell$ tends to infinity.
We are left with the case of subsequences for which $\delta$ is lower bounded by some constant $c\in(0,1)$.
For these $k=k_\ell$ the term
\begin{equation*}
	\frac{d_k\, \ln n^{\mathrm{abs}}\!\left(\e_k/\lambda_1^{d_k/2}, S'_1\right)}{\e_k^{-s}+1} 
	\leq \frac{d_k}{\e_k^{-s}} \, \ln n^{\mathrm{abs}}\left(c, S'_1\right)
	< c' \, \frac{\ln \e_k^{-1}}{\e_k^{-s}}
\end{equation*}
tends to zero, as $\ell \to \infty$, due to \link{lim:ek}.
\end{proof}

\subsubsection{Non-limiting case}
We continue our analysis with necessary and sufficient conditions for $(s,t)$-weak tractability, where $\min\{s,t\}>0$. In view of \autoref{twierdzonko} and \autoref{thm:tensor_t0} it is reasonable to assume that
$\lambda_2>0$ for the remainder of this section. (Otherwise we would have $(0,0)$-weak tractability which in turn shows $(s,t)$-weak tractability for every $s,t>0$.)
In addition, by $m\in\N$ we denote the multiplicity of the first (i.e., largest) eigenvalue of the univariate operator $W_1:=S_1^*\circ S_1$. That is, we assume
\begin{equation*}
	\la_1=\la_2=\ldots=\la_m>\la_{m+1}\geq\la_{m+2}\geq\ldots \geq 0.
\end{equation*}

We start with a characterization for the normalized error criterion:
\begin{theorem}\label{twnormal}
	Let $s,t>0$ and consider a linear tensor product problem $S=(S_d)_{d\in\mathbb{N}}$ with $\la_2>0$ in the 
worst case setting for the normalized error criterion and for the class $\lall$.
	\begin{itemize}
		\item Assume that $m=1$. Then $S$ is $(s,t)$-weakly tractable if and only if
			\begin{equation*}
				\lim_{n\to\infty}\frac{\la_n}{\ln^{-2/s}\! n}=0.
			\end{equation*}
		\item Assume that $m>1$. Then $S$ is $(s,t)$-weakly tractable if and only if
			\begin{equation*}
				t>1
				\quad \text{and} \quad 
				\lim_{n\to\infty}\frac{\la_n}{\ln^{-2/s}\! n}=0.
			\end{equation*}
	\end{itemize}
\end{theorem}

\begin{proof}
First of all note that, without loss of generality, we can assume $\lambda_1=1$. Otherwise we may rescale this sequence according to \link{eq:rescale}; see also \link{eq:explicit_n}, as well as the proof of \cite[Theorem~2.12]{Wei14a}.
This clearly yields that each problem instance $S_d$ is well-scaled, so that
\begin{equation*}
	n^{\mathrm{abs}}(\e,S_d)=n^{\mathrm{norm}}(\e,S_d)
	\qquad \text{for all} \quad d\in\N \quad \text{and} \quad \e\in(0,1].
\end{equation*}
Consequently, we abbreviate the notation and simply write $n(\e,S_d)$ within this proof.

\emph{Step 1 (Necessary conditions).} 
Suppose that $S=(S_d)_{d\in\mathbb{N}}$ is $(s,t)$-weakly tractable (for some non-negative $s$ and $t$) in the sense of \autoref{defi:st-WT}, i.e., assume \link{def:st-weak} to be valid. 
Then the necessity of the limit condition (for all $m\in\N$) immediately follows from \link{cond:1} in \autoref{thm:general} applied for $d=1$. 
Alternatively, we note that \link{def:st-weak} particularly holds for double sequences $((\e_k,d_k))_{k\in\N}\subset(0,1)\times\N$, where $d_k\equiv 1$ and $\e_k\rightarrow 0$, as $k\rightarrow\infty$. This yields
$$
\lim_{\e^{-1}\to \infty} \frac{\ln n(\e,S_1)}{\e^{-s}}=0
$$ 
which in turn is equivalent to $\la_n=o(\ln^{-2/s} n)$, as $n\rightarrow\infty$; see \autoref{lemacik3}.

Moreover, due to the assumption $\lambda_1=1$, we have the following obvious estimate:
$$
	n(\e,S_d)\geq m^d
	\qquad \text{for all fixed} \quad \e<1.
$$
Hence, it holds $\ln n(\e,S_d)\geq d\ln m$, and thus we additionally conclude that $t>1$ if $m>1$.

\emph{Step 2 (Sufficient conditions).} 
We now prove the converse implications.
For this purpose, we assume that $\lambda_{m+1}>0$. 
Note that this can be done without loss of generality, because then the problem only becomes harder (compared to the case $\lambda_{m+1}=0)$.

We will need to estimate the value of 
$$
	n(\e,S_d)
	=\#\! \left\{ (j_1,\dots,j_d)\in\mathbb{N}^d \sep \la_{j_1}\cdot \ldots\ \cdot \la_{j_d} > \e^2 \right\}.
$$
Repeating the combinatorial arguments used in \cite{PapPet09} we obtain the upper bound
$$
	n(\e,S_d)
	\leq \binom{d}{a_d(\e)} \, \left(n(\e^{1/2}, S_1)\right)^{a_d(\e)-1} \, n(\e,S_1) \, d \, m^d,
$$
where
\begin{equation}\label{def:ad}
	a_d(\e) := \min\!\left\{ d,\, \left\lceil 2\, \frac{\ln \e^{-1}}{\ln\la_{m+1}^{-1}} \right\rceil-1 \right\}
	\qquad \text{for} \quad \e<1 \quad \text{and} \quad d\in\N.
\end{equation}
In comparison with the estimate used in \cite{PapPet09} there are two differences: the $m^d$ factor 
and the appearance of $\la_{m+1}$ instead of $\la_2$ in \link{def:ad}. 
They simply stem from the fact that now we have, in general, $m$ indices $j\in\mathbb{N}$ corresponding to 
eigenvalues $\la_j=1$. Clearly, if $m=1$, then our estimate is the same as in \cite{PapPet09}.

Note that $a_d(\e)=\Theta\!\left(\min\{d,\,\ln\e^{-1}\}\right)$, where the equivalence factors in the 
$\Theta$-notation depend on $\la_{m+1}$, but not on $\e$ or $d$.
The logarithm of $n(\e,S_d)$ 
can be bounded, as in \cite{PapPet09}, from above by

\begin{equation}\label{est:combinatoric}
	\ln n(\e,S_d)
	\leq a_d(\e) \, \ln d + a_d(\e) \, \ln n(\e^{1/2}, S_1) + \ln n(\e,S_1) + \ln d + d \, \ln m.
\end{equation}
Next let us define
\begin{eqnarray*}
	\alpha \!\!\!&:=&\!\!\! \limsup_{\e^{-1}+d\to\infty} \frac{\ln n(\e,S_d)}{\e^{-s}+d^{\,t}}, \\
	\beta \!\!\!&:=&\!\!\! \limsup_{\e^{-1}+d\to\infty} \left[\frac{a_d(\e) \, \ln d}{\e^{-s}+d^{\,t}}+
\frac{a_d(\e) \, \ln n(\e^{1/2}, S_1)}{\e^{-s}+d^{\,t}}+ \frac{\ln n(\e, S_1)}{\e^{-s}+d^{\,t}}+
\frac{\ln d}{\e^{-s}+d^{\,t}}\right],\quad \text{and}\\
	 \gamma \!\!\!&:=&\!\!\! \limsup_{\e^{-1}+d\to\infty} \frac{d \, \ln m}{\e^{-s}+d^{\,t}}.
\end{eqnarray*}
Then \link{est:combinatoric} yields
$$
	0\leq \alpha \leq \beta + \gamma
$$
and thus it suffices to prove that $\beta = \gamma = 0$ in order to show the claim.

\emph{Substep 2.1.}
Here we show that $\beta = 0$ for all $s,t>0$. 
For this purpose, let 
\begin{equation*}
	x(\e,d):=\max\{d,\,\e^{-1}\}.
\end{equation*}
Then, obviously, $\ln d \leq \ln x(\epsilon,d)$ and there exists a constant $c>0$ such that
$$
	\e^{-s}+d^{\, t}
	\geq \e^{-\min\{s,t\}}+d^{\min\{s,t\}}
	\geq c \, (\e^{-1}+d)^{\min\{s,t\}} 
	\geq c \, x(\epsilon,d)^{\min\{s,t\}}, 
$$
as well as
$$
	a_d(\e)
	\leq c \, \min\{d,\, \ln \e^{-1}\}
	\leq c \, \ln\e^{-1}
	\leq c \, \ln x(\epsilon,d).
$$ 
Since $x(\epsilon,d)\rightarrow\infty$, as $\e^{-1}+d\rightarrow \infty$, we thus have
$$
	\limsup_{\e^{-1}+d \to \infty} \frac{a_d(\e) \, \ln d}{\e^{-s}+d^{\,t}}
	\leq \limsup_{\e^{-1}+d \to \infty} \frac{\ln^2 x(\epsilon,d)}{x(\epsilon,d)^{\min\{s,t\}}}=0.
$$

Moreover, let $\delta:=\delta(\e,d):=(\e^{-s}+d^{\,t})^{-1/s}$. 
Then $\delta\to 0$ if and only if $\e^{-1}+d\to\infty$. 
Consequently \autoref{lemacik3} implies
\begin{equation*}
	\limsup_{\e^{-1}+d \to \infty}\frac{\ln n(\e,S_1)}{\e^{-s}+d^{\,t}}
	\leq \limsup_{\e^{-1}+d \to \infty}\frac{\ln n([\e^{-s}+d^{\,t}]^{-1/s},S_1)}{\e^{-s}+d^{\,t}}
	= \limsup_{\delta \to 0} \frac{\ln n(\delta,S_1)}{\delta^{-s}}
	=0,
\end{equation*}
as well as
\begin{align*}
	\limsup_{\e^{-1}+d \to \infty} \frac{a_d(\e) \, \ln n(\e^{1/2}, S_1)}{\e^{-s}+d^{\,t}}
	&\leq c \, \limsup_{\e^{-1}+d \to \infty}\frac{\ln\e^{-1} \, \ln n(\e^{1/2},S_1)}{\e^{-s}+d^{\,t}} \\
	&\leq c \, \limsup_{\e^{-1}+d \to \infty} \frac{\ln\delta^{-1} \, \ln n(\delta^{1/2},S_1)}{\delta^{-s}} \\
	&= c \, \limsup_{\delta \to 0} \frac{\ln\delta^{-1}}{(\delta^{-1})^{s/2}} \cdot \frac{\ln n(\delta^{1/2},S_1)}{(\delta^{1/2})^{-s}}
	=0.
\end{align*}

In addition, $d\leq \delta^{-s/t}$ gives
\begin{equation*}
	\limsup_{\e^{-1}+d \to \infty}\frac{\ln d}{\e^{-s}+d^{\,t}}
	\leq \limsup_{\delta \to 0} \frac{\ln \delta^{-s/t}}{\delta^{-s}}
	=0
\end{equation*}
such that $\beta=0$, as claimed.

\emph{Substep 2.2.}
It remains to show that $\gamma=0$.
If $m=1$, then this is obvious since $\ln m=0$.
If $m>1$, then 
\begin{equation*}
	\limsup_{\e^{-1}+d \to \infty} \frac{d \, \ln m}{\e^{-s}+d^{\,t}}
	\leq \ln m \cdot \limsup_{\delta \to 0} \frac{\delta^{-s/t}}{\delta^{-s}}
	=0,
\end{equation*}
due to the additional assumption $t>1$.
\end{proof}

We complete our investigations of the non-limiting case of $(s,t)$-weak tractability for linear tensor product problems with the following results for the absolute error criterion.

\begin{theorem}\label{twabs}
	Let $s,t> 0$ and consider a linear tensor product problem $S=(S_d)_{d\in\mathbb{N}}$ with $\la_2>0$ in the 
worst case setting for the absolute error criterion and for the class $\lall$.
	\begin{itemize}
		\item Let $\la_1<1$. Then $S$ is $(s,t)$-weakly tractable if and only if
			\begin{equation}\label{cond:log_decay}
				\lim_{n\to\infty}\frac{\la_n}{\ln^{-2/s} \! n}=0.
			\end{equation}
		\item Let $\la_1=1$ and 
			\begin{itemize}
				\item[1.)] assume that $m=1$. Then $S$ is $(s,t)$-weakly tractable if and only if
					\begin{equation*}
						\lim_{n\to\infty}\frac{\la_n}{\ln^{-2/s} \! n}=0.
					\end{equation*}
				\item[2.)] assume that $m>1$. Then $S$ is $(s,t)$-weakly tractable if and only if
					\begin{equation*}
						t>1
						\quad \text{and} \quad 
						\lim_{n\to\infty}\frac{\la_n}{\ln^{-2/s} \! n}=0.
					\end{equation*}
			\end{itemize}
		\item Let $\la_1>1$ and define $S_1':=\frac{1}{\sqrt{\lambda_1}} S_1$. Then $(s,t)$-weak tractability of $S$ implies
	\begin{equation}\label{ness_cond}
		t>1
		\quad \text{and} \quad 
		\lim_{\e^{-1}+d\to\infty} \frac{\max_{\ell=1,\ldots,d}\limits \left[ \ell \cdot \ln n^{\mathrm{abs}}\!\left( ( \e / \la_1^{d/2} )^{1/\ell},\, S'_1 \right) \right]}{\e^{-s}+d^{\,t}} 
		= 0.
	\end{equation}
		Moreover, the conditions
	\begin{equation}\label{suff_cond}
		\quad t>1
		\quad \text{and} \quad 
		\lim_{\e^{-1}+d\to\infty} \frac{\ln d \cdot \max_{\ell=1,\ldots,d}\limits \left[ \ell \cdot \ln n^{\mathrm{abs}}\!\left( ( \e / \la_1^{d/2} )^{1/\ell},\, S'_1 \right) \right]}{\e^{-s}+d^{\,t}}
		= 0
	\end{equation}
	are sufficient for $S$ to be $(s,t)$-weakly tractable.
	\end{itemize}
\end{theorem}

\begin{proof}
\emph{Step 1 (Case $\lambda_1=1$).} 
As explained in the proof of \autoref{twnormal}, for this case the results w.r.t.\ the absolute and the normalized error criteria coincide. Hence, the assertion follows from \autoref{twnormal}.

\emph{Step 2 (Case $\lambda_1<1$).}
Similar to the proof for the normalized error criterion, necessity of \link{cond:log_decay} follows from \link{cond:1} in \autoref{thm:general}.

To see that this limit condition is also sufficient, we note that the linear tensor product problem $S'':=(S''_d)_{d\in\mathbb{N}}$ defined by the sequence
\begin{equation*}
	\la''_1 := 1
	\qquad \text{and} \qquad 
	\la''_n: = \la_n
	\quad \mbox{for} \quad n>1
\end{equation*}
(of eigenvalues of the univariate operator $W''_1:=(S''_1)^{*} \circ S''_1$) is harder than the problem $S$.

\emph{Step 3 (Case $\lambda_1>1$).}
\emph{Substep 3.1 (Necessity).}
Suppose that $S=(S_d)_{d\in\mathbb{N}}$ is $(s,t)$-weakly tractable, i.e., assume \link{def:st-weak}.
From \cite[Theorem 5.5]{NW08} we know that $S$ suffers from the curse of dimensionality. 
Hence, for all $\e_0\in(0,1)$ there exists $c>0$ such that
\begin{equation*}
	n^{\mathrm{abs}}(\e_0,S_d) \geq (1+c)^d
	\qquad \text{for all} \quad d\in\N.
\end{equation*}
Considering the double sequence $((\e_k,d_k))_{k\in\N}$ with $\e_k\equiv \e_0$ in \link{def:st-weak} thus shows that $t>1$ since otherwise
\begin{equation*}
	\frac{\ln n^{\mathrm{abs}}(\e_k,S_{d_k})}{\e_k^{-s}+d_k^{\,t}} \geq \frac{d_k \, \ln (1+c)}{\e_0^{-s}+d_k^{\,t}} \geq c' \, d_k^{1-t}
\end{equation*}
does not tend to zero, as $k\to\infty$.

Furthermore, consider the linear tensor product problem $S':=(S'_d)_{d\in\N}$ defined by
$S'_1:=\frac{1}{\sqrt{\la_1}}S_1$. 
Then the ordered eigenvalues $\la'_n$, $n\in\N$, of $W'_1 := (S'_1)^*\circ S'_1$ satisfy \link{eq:rescale} and \link{eq:n_modified} implies that for all $\ell=1,\ldots,d$ it holds
\begin{align*}
	n^{\mathrm{abs}}(\e, \, S_d) 
	&= \#\!\left\{ (j_1,\ldots,j_d) \in\mathbb{N}^d \sep \la'_{j_1}\cdot\ldots\cdot\la'_{j_d}>\e^2/\la_1^d \right\} \\
	&\geq \#\!\left\{ (j_1,\ldots,j_\ell) \in\mathbb{N}^\ell \sep \la'_{j_1}\cdot\ldots\cdot\la'_{j_\ell}> \left( \prod_{k=1}^\ell \left[ \e/\la_1^{d/2} \right]^{1/\ell} \right)^{2} \right\} \\
	&\geq \# \bigtimes_{k=1}^\ell \!\left\{ j\in\mathbb{N} \sep \la'_j > \left( \left[ \e/\la_1^{d/2} \right]^{1/\ell} \right)^{2} \right\} \\
	&= n^{\mathrm{abs}}\!\left(\left(\e/\la_1^{d/2}\right)^{1/\ell}, \, S'_1\right)^\ell.
\end{align*}
Hence, $(s,t)$-weak tractability also implies
\begin{equation*}
	\frac{\max_{\ell=1,\ldots,d}\limits \left[ \ell \cdot \ln n^{\mathrm{abs}}\!\left( \left( \e/\la_1^{d/2} \right)^{1/\ell}, \, S'_1\right) \right]}{\e^{-s}+d^{\,t}} \rightarrow 0, 
	\qquad \text{as} \quad \e^{-1}+d\to\infty.
\end{equation*}

\emph{Substep 3.1 (Sufficiency).}
To see that the stated conditions are sufficient for $(s,t)$-weak tractability, we employ \autoref{lemacikgorny} to obtain
\begin{align*}
		\ln n^{\mathrm{abs}}(\e, \, S_d)
		&\leq d\, \ln d + \sum_{\ell=1}^d \ln n^{\mathrm{abs}}\!\left( ( \e / \la_1^{d/2} )^{1/\ell},\, S'_1 \right)\\
		&\leq d\, \ln d + \max_{\ell=1,\ldots,d} \left[ \ell \cdot \ln n^{\mathrm{abs}}\!\left( ( \e / \la_1^{d/2} )^{1/\ell},\, S'_1 \right) \right] \sum_{\ell=1}^d \frac{1}{\ell} \\
		&\leq d\, \ln d + c\cdot \ln d \cdot \max_{\ell=1,\ldots,d} \left[ \ell \cdot \ln n^{\mathrm{abs}}\!\left( ( \e / \la_1^{d/2} )^{1/\ell},\, S'_1 \right) \right]
\end{align*}
with some $c>0$.

Due to the assumption $t>1$, we obviously have
$$
	\lim_{\e^{-1}+d\rightarrow\infty}\frac{d\ln d}{\e^{-s}+d^{\,t}}=0.
$$
As the convergence of the remaining term directly follows as well, we have completed the proof.
\end{proof}

\begin{remark}
Let us add some final remarks on the case $\lambda_1>1$ in \autoref{twabs}:
\begin{itemize}
	\item[(i)] Note that, in contrast to the case $\lambda_1 \leq 1$, for $\lambda_1>1$ our conditions are stated in terms of maxima of weighted univariate information complexities $n^{\mathrm{abs}}(( \e/\la_1^{d/2} )^{1/\ell}, \, S'_1)$.
	For every fixed $\ell$ the growth behavior of these information complexities can be translated into decay conditions of singular values $\lambda^{(1)}=(\lambda_n)_{n\in\N}$: 
	
	\begin{lemma}\label{lemacik4}
		Let $L>1$. Then, for every $t>0$,
		\begin{equation*}
			\lim_{n\rightarrow\infty}\frac{\ln\la_n^{-1}}{\ln^{1/t} n} = \infty
			\qquad \text{if and only if} \qquad
			\lim_{k\to \infty} \frac{\ln n^{\mathrm{abs}}\!\left(1/L^k, \,S_1\right)}{k^{\,t}}=0.
		\end{equation*}		
	\end{lemma}
	
	Before we give a proof of this lemma, let us remark that
	this can be used to show that $\lim_{n\rightarrow\infty}\frac{\ln\la_n^{-1}}{\ln^{1/t} n} = \infty$ is necessary for $(s,t)$-weak tractability if $\lambda_1>1$. 
	To see this, note that choosing $\ell=1$ in \link{ness_cond} particularly implies
	\begin{equation*}
		\lim_{\e^{-1}+d\to\infty} \frac{ \ln n^{\mathrm{abs}}\!\left( \e / \la_1^{d/2} ,\, S'_1 \right)}{\e^{-s}+d^{\,t}} 
		= 0.
	\end{equation*}
	Hence, considering the sequence $((\e_k,d_k))_{k\in\N}$ with $\e_k\equiv 1$ and $d_k=k$ shows that we have $\ln n^{\mathrm{abs}}\!\left(1/L^k, S'_1\right)=o(k^{\,t})$, where $L:=\lambda_1^{1/2}>1$. This in turn yields 
\begin{equation*}
	\lim_{n\rightarrow\infty} \frac{\ln\la_n^{-1}}{\ln^{1/t} \! n}
	= \lim_{n\rightarrow\infty} \frac{\ln(\la'_n)^{-1} - \ln \lambda_1}{\ln^{1/t} \! n} = \infty,
\end{equation*}
due to \autoref{lemacik4} (applied for $S_1':=\frac{1}{\sqrt{\lambda_1}} \, S_1$).

We note in passing that this necessary condition is much stronger than the decay condition \link{cond:log_decay} which characterizes $(s,t)$-weak tractability in the case $\lambda_1\leq 1$. 
In addition, it is interesting to see that it involves the parameter $t$. (In contrast \link{cond:log_decay} only depends on~$s$!)

	\begin{proof}[Proof of \autoref{lemacik4}]
	\emph{Step 1 (Sufficiency).}
It follows from the assumption that for every $c>0$ there exists 
$k_c\in\mathbb{N}$ such that for every $k\geq k_c$ we have
$$	
	n^\mathrm{abs}(1/L^k,S_1)
	\leq \left\lceil \exp\!\left( c\, k^{\, t} \right) \right\rceil.
$$
Let us fix $c>0$. 
Hence, for $k\geq k_c$ we have $\la_{\left\lceil \exp\left( c\, k^{\, t} \right) \right\rceil+1}\leq 1/L^{2k}$.
Now set
$$
	m_k
	:= \left\lceil \exp\!\left( c\, k^{\, t} \right) \right\rceil+1
$$
and observe that $m_k$ monotonically tends to infinity, as $k\rightarrow \infty$.
In addition, we see that from $m_k\leq \exp\!\left( c\, k^{\, t} \right)+2$
it follows that 
$$
	\left[\frac{\ln(m_k-2)}{c}\right]^{1/t}
	\leq k,
$$
Hence, for $k\geq k_c':=\max\{k_c,\, \min\{k\in\N \sep m_k>2\} \}$, we obtain the estimate
$$
	\la_{m_k}\leq(1/L^k)^2 
	\leq L^{-2\,[ (\ln(m_k-2))/c ]^{1/t}},
$$
due to the ordering of $(\lambda_j)_{j\in\N}$ and the assumption that $L>1$.
Now it is easy to see that for 
$k\geq k_c^{''}:=\max\{k_c', \, \min\{k\in\N \sep \left[(\ln(m_k-2))/(\ln m_k)\right]^{1/t}\geq 1/2\}\}$
we have
$$
	\frac{\ln \la_{m_k}^{-1}}{\ln^{1/t} m_k} 
	\geq \frac{\ln L}{c^{1/t}}.
$$

For $n\geq k_c^{''}$ let us define $k(n):=\max\{k\in\N \sep m_k\leq n\}$.
Then, clearly, $m_{k(n)}\leq n<m_{k(n)+1}$ and $\lambda_{n} \leq \lambda_{m_{k(n)}}$, so that
$$
	\frac{\ln\la_n^{-1}}{\ln^{1/t} n}
	> \frac{\ln\la_{m_{k(n)}}^{-1}}{\ln^{1/t} m_{k(n)+1}}
	=\frac{\ln\la_{m_{k(n)}}^{-1}}{\ln^{1/t} m_{k(n)}}
	\left(\frac{\ln m_{k(n)}}{\ln m_{k(n)+1}}\right)^{1/t}
	\quad \text{for all} \quad n\geq k_c^{''}.
$$
Observe that for sufficiently large $k$, say for $k\geq K_c$, we have
$$
	\frac{\ln m_k}{\ln m_{k+1}}
	\geq \frac{c\, k^{\, t}}{2\,c\,(k+1)^t}
	=\frac{1}{2}\left(\frac{k}{k+1}\right)^t
	\geq\frac{1}{4}.
$$
This implies that for all $n\in\mathbb{N}$ such that $n\geq k_c^{''}$ and $k(n)\geq \max(k_c^{''}, \, K_c)$ 
we have 
$$
	\frac{\ln \la_n^{-1}}{\ln^{1/t} n}
	> \frac{\ln L}{(4 \, c)^{1/t}}.
$$
Hence, for every fixed $c>0$ the last inequality holds for all but a finite number of 
natural numbers~$n$. In conclusion, this shows the ``if''-part of the assertion.

\emph{Step 2 (Necessity).}
Note that $L>1$ implies that $1/L^k$ tends to zero, as $k\to\infty$.
Thus, w.l.o.g.\ we can assume that $n^{\mathrm{abs}}((1/L)^k,S_1)$ grows without bound, as $k$ approaches infinity (otherwise there is nothing to show). 
Now observe that our assumption is equivalent to 
\begin{equation*}
	\lim_{n\rightarrow\infty} \frac{\ln n}{\ln^t \! \la_n^{-1}}
	= 0,
\end{equation*}
i.e., $\ln n = o\!\left(\ln^t \la_n^{-1}\right)$, as $n\rightarrow\infty$.
Consequently, we obtain
\begin{equation}\label{est:nS1}
	\ln n^{\mathrm{abs}} \! \left( 1/L^k, \, S_1 \right) 
	= o\! \left( \ln^t \! \lambda_{n^{\mathrm{abs}}(1/L^k, \, S_1)}^{-1} \right), 
	\qquad \text{as} \quad 	k \rightarrow \infty.
\end{equation}
In addition, the general relation \link{eq:explicit_n} applied for $S_1$ yields $1/L^{2k} < \lambda_{n^{\mathrm{abs}}(1/L^k, \, S_1)}$ for all $k\in\N$. 
Hence, we can estimate
\begin{align*}
	\ln^t \! \lambda_{n^{\mathrm{abs}}(1/L^k, \, S_1)}^{-1}
	< \ln^t \! L^{2k} = c\, k^t
\end{align*}
with $c=(2 \ln L)^t > 0$ independent $k$.
Combining this with \link{est:nS1} finally proves the ``only~if''-part of \autoref{lemacik4}.
\end{proof}
	 
	\item[(ii)] We can get rid of the additional logarithm in the second part of the sufficient condition \link{suff_cond} at the expense of a slightly larger power of $\ell$ in the maximum. That is, the limit condition can be replaced by
		\begin{equation*}
			\lim_{\e^{-1}+d\to\infty} \frac{\max_{\ell=1,\ldots,d}\limits \left[ \ell^{\,p} \cdot \ln n^{\mathrm{abs}}\!\left( ( \e / \la_1^{d/2} )^{1/\ell},\, S'_1 \right) \right]}{\e^{-s}+d^{\,t}}
		= 0
		\qquad \text{with} \quad 1<p<t.
		\end{equation*}
\end{itemize}
\end{remark}

\subsection{The power of function values for multivariate approximation}\label{subsect:standard}
This section is based on ideas from \cite{NW12} 
which relate the power of general linear information (from the class $\lall$) with the power of function values
($\lstd$) for certain multivariate approximation problems.

Let $H_d$ be some non-trivial, separable reproducing kernel Hilbert space 
of $d$-variate functions defined (almost everywhere) 
on a set $D_d\subset\mathbb{R}^d$ of positive Lebesgue measure. 
Moreover, for every $d\in\mathbb{N}$ let the target space $G_d:=L_2(D_d,\rho_d)$
be the space of functions over $D_d$ which are square-integrable w.r.t.\ some probability density
$\rho_d$. 
For all $d$ we assume that $H_d$ is continuously embedded
in $G_d$. Then  
the multivariate approximation problem 
${\rm APP}:=({\rm APP}_d)_{d\in\mathbb{N}}$ is 
defined by
\begin{equation}\label{def:APP}
	{\rm APP}_d\colon H_d\rightarrow G_d, \qquad \text{where} \quad {\rm APP}_d f:=f \quad \text{for all}\quad f\in H_d, \, d\in\N.
\end{equation}

For our purposes, the most relevant results from \cite{NW12} for this problem are for weak tractability. 
It turns out that in many cases weak tractability for the class
$\lall$ implies weak tractability for the class $\lstd$.  
Interestingly enough, the same proofs can be also applied for
$(s,t)$-weak tractability since they rely on estimates of the form 
\begin{equation*}
	n^{\mathrm{crit}}(\e,{\rm APP}_d;\lstd) 
	\leq n^{\mathrm{crit}}(\e/C,{\rm APP}_d;\lall) \cdot r^{\mathrm{crit}}(\e,{\rm APP}_d),
	\qquad {\mathrm{crit}} \in \{\mathrm{abs},\mathrm{norm}\},
\end{equation*}
for some $C\geq 1$ and a known function $r^{\mathrm{crit}}(\e,{\rm APP}_d)$.
Therein $n^{\mathrm{crit}}(\e,{\rm APP}_d;\Lambda)$ denotes the information complexity of ${\rm APP}$ with respect to information from the class $\Lambda\in\{\lstd,\lall\}$ in the worst case setting for the absolute or normalized error criterion, respectively.
Hence, we can present relations
between $(s,t)$-weak tractability for the classes $\lall$ and $\lstd$ 
with very brief proofs which will allow us to keep this section short.

It is known that if the operator $W_d:={\rm APP}_d^* \circ {\rm APP}_d$ 
has infinite trace for some $d\in\mathbb{N}$, 
then there is no non-trivial relation between 
tractabilities for the classes $\lall$ and $\lstd$; 
see \cite[Corollary~26.2]{NW12}. Therefore we assume that 
\begin{equation*}
	{\rm trace}(W_d):=\sum_{j=1}^{\dim H_d} \lambda_{d,j} < \infty
	\qquad \text{for every}\quad d\in\mathbb{N},
\end{equation*}
where $\lambda^{(d)}:=(\lambda_{d,j})_{j=1}^{\dim H_d}$ again denotes the ordered sequence of eigenvalues of $W_d$.
In addition, let us assume that ${\rm trace}(W_d)>0$ for all $d\in\N$, i.e., suppose that ${\rm APP}$ is not trivial.
As before, we let ${\rm CRI}_d$ be defined by~\link{def:CRI}.
Note that the finite trace condition immediately implies that 
\begin{equation*}
	\lambda_{d,j} \in \mathcal{O}(j^{-1}), 
	\qquad \text{as} \quad j\to\infty,
\end{equation*}	
which is a much stronger condition than $\lambda_{d,j} \in o(\ln^{-2/s}\! j)$ which was needed for $(s,t)$-weak tractability with $s>0$; cf.\ \link{cond:1} in \autoref{thm:general}.

\begin{theorem}\label{twierdzenieapp}
	Consider the multivariate approximation problem \link{def:APP} w.r.t.\ the worst case setting for the absolute or normalized error criterion. 
	Let $s,t>0$ and assume that the trace of~$W_d$ is finite and non-trivial for all $d\in\mathbb{N}$. 
	If, additionally,
	\begin{equation}\label{cond:trace}
		\lim_{d\to\infty}\frac{\ln\ceil{ {\rm trace}(W_d)/{\rm CRI}_d }}{d^{\,t}}=0,
	\end{equation}
	then $(s,t)$-weak tractabilities of ${\rm APP}$ 
	for the classes $\lall$ and $\lstd$ are equivalent.
\end{theorem}

\begin{proof}
This theorem corresponds to \cite[Theorem 26.11]{NW12} for weak tractability. 
Since in our setting $\lstd\subset\lall$, it is enough to show that $(s,t)$-weak tractability for the class $\lall$ implies $(s,t)$-weak tractability for the class $\lstd$. 
This implication holds because, as shown in the proof of \cite[Theorem 26.11]{NW12}, we have 
\begin{equation*}
	n^{\mathrm{crit}}(\e,{\rm APP}_d;\lstd)
	\leq n^{\mathrm{crit}}(\e/\sqrt{2},{\rm APP}_d;\lall) \cdot 4\, \e^{-2} \, \ceil{\frac{{\mathrm{trace}}(W_d)}{{\rm CRI}_d}},
	\qquad \mathrm{crit}\in\{\mathrm{abs}, \mathrm{norm}\},
\end{equation*}
for all $\e\in(0,1)$ and $d\in\N$.
Therefore,
\begin{align*}
	\frac{\ln n^{\mathrm{crit}}(\e,{\rm APP}_d;\lstd)}{\e^{-s}+d^{\,t}}
	&\leq \frac{\ln n^{\mathrm{crit}}(\e/\sqrt{2},{\rm APP}_d;\lall)}{2^{-s/2}\, (\e/\sqrt{2})^{-s} + d^{\,t}} 
		+ \frac{\ln 4 + 2 \, \ln\e^{-1}}{\e^{-s}+d^{\,t}} 
		+ \frac{\ln\ceil{ {\rm trace}(W_d)/{\rm CRI}_d }}{\e^{-s}+d^{\,t}} 
\end{align*}
tends to zero, as $\e^{-1}+d$ approaches infinity, since we assumed that $s,t>0$.
\end{proof}

\begin{remark}
We note in passing that for $\min\{s,t\}=0$ the above proof fails: if $s=0$, then the term
\begin{equation*}
	\frac{\ln 4 + 2 \, \ln\e^{-1}}{\e^{-s}+d^{\,t}} 
\end{equation*}
explodes for double sequences $((\e_k,d_k))_{k\in\N}$ with $\e_k^{-1}+d_k\to\infty$ and uniformly bounded $d_k$.
Moreover, for $t=0$ and double sequences with constant $\e_k$ it does not tend to zero as well.
\end{remark}

As a consequence of \autoref{twierdzenieapp} we conclude the following corollary for the case where $\mathrm{APP}$ is a linear tensor product problem:

\begin{corollary}
	Let $s,t>0$ and consider the multivariate approximation problem for tensor product source spaces $H_d(D_d)=\bigotimes_{\ell=1}^d H_1(D_1)$ and tensor product densities $\rho_d=\bigotimes_{\ell=1}^d \rho_1$ in the worst case setting. Furthermore assume $0<{\rm trace}(W_1)<\infty$. Then
	\begin{itemize}
		\item for the absolute error criterion $(s,t)$-weak tractability w.r.t.\ the class $\lall$ is equivalent to $(s,t)$-weak tractability for the class $\lstd$ if
			\begin{equation*}
				t>1 \quad \text{or} \quad {\rm trace}(W_1) \leq 1.
			\end{equation*}
		\item for the normalized error criterion $(s,t)$-weak tractability w.r.t.\ the class $\lall$ is equivalent to $(s,t)$-weak tractability for the class $\lstd$ if
			\begin{equation*}
				t>1 \quad \text{or} \quad \lambda_2 = 0.
			\end{equation*}
	\end{itemize}
\end{corollary}
\begin{proof}
To prove the claim we like to apply \autoref{twierdzenieapp}. 
For this purpose, we need to check~\link{cond:trace}.
Note that for linear tensor product problems it holds
\begin{equation*}
	{\rm trace}(W_d)=({\rm trace}(W_1))^{\,d}
	\qquad \text{and} \qquad {\rm CRI}_d={\rm CRI}_1^{\,d}.
\end{equation*}
Hence, 
$$
	\lim_{d\to\infty}\frac{\ln \ceil{ {\rm trace}(W_d)/{\rm CRI}_d } }{d^{\,t}}
	=\lim_{d\to\infty} d^{1-t} \, \ln \ceil{ {\rm trace}(W_1)/{\rm CRI}_1 }
$$
equals zero if and only if $t>1$ or ${\rm trace}(W_1)/{\rm CRI}_1 \leq 1$. (Remember that ${\rm trace}(W_1)>0$!)
For the absolute error criterion the latter condition is satisfied if the trace of $W_1$ is bounded by one while for the normalized error criterion $\lambda_2$ needs to be zero.
\end{proof}

We conclude the discussion with the observation that the conditions \link{cond:trace} and $t>1$, respectively, are sharp in the sense that they cannot be dropped in general. 
To prove this, we refer to an example given in \cite[Subsection 26.4.1]{NW12}. 
Therein $H_d$ coincides with the $d$-fold tensor product of some univariate Korobov space defined on $D_1:=[0,1]$. For the sequence of univariate eigenvalues it holds
\begin{equation*}
	0 < \lambda_2 < \lambda_1 = 1 < \mathrm{trace}(W_1)<\infty
\end{equation*}
and the density $\rho_1$ is assumed to be identically $1$.
Then it can be shown that the approximation problem under consideration is quasi-polynomially tractable w.r.t.\ the class $\Lambda^{\mathrm{all}}$ while it suffers from the curse of dimensionality when dealing with information from $\Lambda^{\mathrm{std}}$. This yields that we have $(s,t)$-weak tractability w.r.t.\ $\Lambda^{\mathrm{all}}$ for all $s,t>0$ while it holds $(s,t)$-weak tractability w.r.t.\ $\Lambda^{\mathrm{all}}$ only if \link{cond:trace} is satisfied, i.e., if $t>1$.

\section{Examples}\label{sect:applications}
In this final section we illustrate our new notion of $(s,t)$-weak tractability by means of two more or less classical problems which recently attracted some attention in information-based complexity.
In \autoref{sect:sobolev} we deal with approximation problems of Sobolev embeddings, whereas \autoref{sect:int} is concerned with the integration problem for a class of smooth functions.

\subsection{Approximation of Sobolev embeddings}\label{sect:sobolev}
We follow the lines of \cite{KueSicUll14} and consider the approximation problem
\begin{gather}\label{prob:id1}
	\id_d \colon H^{\alpha,\square} (\mathbb{T}^d) \nach L_2(\mathbb{T}^d), \quad d\in\N,
\end{gather}
w.r.t.\ the worst case setting and information from $\Lambda^{\mathrm{all}}$.
Therein $H^{\alpha,\square} (\mathbb{T}^d)$ denotes the isotropic Sobolev spaces of periodic functions over the $d$-dimensional torus $\mathbb{T}^d=[0,2\pi]^d$ with smoothness $\alpha \geq 0$.
We use the symbol $\square\in\{+,*,\sharp\}$ to distinguish the following (equivalent) norms based on Fourier coefficients 
\begin{equation*}
	c_k(f) := \frac{1}{(2\pi)^{d/2}} \int_{\mathbb{T}^d} f(x)\, \exp\left(-\i\, kx\right) \d x, \qquad k\in\Z^d,
\end{equation*}
see \cite[Definition~2.2]{KueSicUll14} for details:
\begin{itemize}
\item natural norm (for $\alpha\in\N$ all derivatives of order at most $\alpha)$:
\begin{gather*}
	\norm{f \sep H^{\alpha,+} (\mathbb{T}^d)}
	:= \left[  \sum_{k\in\Z^d} \abs{c_k(f)}^2 \left( 1+ \sum_{j=1}^d \abs{k_j}^2 \right)^\alpha \right]^{1/2},
\end{gather*}
\item modified natural norm (for $\alpha \in \N$ $L_2$-norm and highest order derivatives):
\begin{gather*}
	\norm{f \sep H^{\alpha,*} (\mathbb{T}^d)}:= \left[  \sum_{k\in\Z^d} \abs{c_k(f)}^2 \left( 1+ \sum_{j=1}^d \abs{k_j}^{2\alpha} \right) \right]^{1/2},
\end{gather*}
\item auxiliary norm:
\begin{gather*}
	\norm{f \sep H^{\alpha,\sharp} (\mathbb{T}^d)}:= \left[  \sum_{k\in\Z^d} \abs{c_k(f)}^2 \left( 1+ \sum_{j=1}^d \abs{k_j} \right)^{2\alpha} \right]^{1/2}.
\end{gather*}
\end{itemize}
The norm in the target space of square-integrable functions on the $d$-dimensional torus, $L_2(\mathbb{T}^d)$, is simply given by
\begin{equation*}
	\norm{f \sep L_2(\mathbb{T}^d)}
	:=\left[  \sum_{k\in\Z^d} \abs{c_k(f)}^2 \right]^{1/2},
\end{equation*}
i.e., for $d\in\N$ we set $L_2(\mathbb{T}^d):=H^{0,+}(\mathbb{T}^d)=H^{0,*}(\mathbb{T}^d)=H^{0,\sharp}(\mathbb{T}^d)$.

The main aim of \cite{KueSicUll14} was it to 
investigate sharp bounds on the corresponding \emph{approximation numbers}, defined by
\begin{equation*}
	a_{n,d} 
	:= a_{n,d}^{\alpha,\square} 
	:= \inf_{\rank{A}<n} \sup_{\norm{f\sep H^{\alpha,\square}(\mathbb{T}^d)} \leq 1} \norm{\id_d f - A f \sep L_2(\mathbb{T}^d)}, 
	\qquad n,d\in\N,
\end{equation*}
where $\alpha>0$ and $\square\in\{+,*,\sharp\}$ are assumed to be chosen fixed.
As explained already in the introduction in IBC we usually study the closely related \emph{$n$th minimal (worst case) error}, given by
\begin{equation}\label{eq:e_and_a}
	e(n,d) = e^{\alpha,\square}(n,d) = a_{n+1,d}^{\alpha,\square}
	\qquad \text{for all} \qquad n\in\N_0, \quad d\in\N,
\end{equation}
as well as the corresponding information complexity $n^{\mathrm{crit}}(\epsilon,\id_d)$ defined in \link{def:abs_crit} and \link{def:norm_crit}, respectively.
In the following we slightly abuse the notation and write $n^{\alpha,\square}(\epsilon,\id_d)$ for these information complexities. Note that this is reasonable because for this problem we do not need to distinguish between the absolute and the normalized error criterion since
the initial error $\epsilon_d^{\mathrm{init}}$ equals $1$ for all $d\in\N$, every $\alpha > 0$, and each $\square\in\{+,*,\sharp\}$.

Among other things the authors of \cite{KueSicUll14} found that these three approximation problems never suffer from the curse of dimensionality; cf. \cite[Theorem~5.6]{KueSicUll14}.
Moreover, they investigated (almost sharp) conditions on $\alpha$ such that weak tractability holds.
We extend their results by proving the following

\begin{theorem}\label{thm:sobolev}
Let $\alpha>0$ and consider the approximation problem defined in \link{prob:id1}.
\begin{itemize}
	\item In the case $\square=+$ the problem is $(s,t)$-weakly tractable if and only if 
			\begin{equation*}
			s > \frac{2}{\alpha} \,\text{ and }\, t>0 
			\qquad \text{or} \qquad
			s > 0 \,\text{ and }\, t>1. 
		\end{equation*}
	\item In the case $\square=*$
	\begin{itemize}
		\item[1.)] $(s,t)$-weak tractability implies that 
			\begin{equation*}
				s > 2 \,\text{ and }\, t>0 
				\qquad \text{or} \qquad
				s > 0 \,\text{ and }\, t>1. 
			\end{equation*}
		\item[2.)] the conditions 
			\begin{equation*}
				s > \max\!\left\{2, \frac{1}{\alpha} \right\} \,\text{ and }\, t>0 
				\qquad \text{or} \qquad
				s > 0 \,\text{ and }\, t>1 
			\end{equation*}		 
			imply $(s,t)$-weak tractability.
	\end{itemize}
	\item In the case $\square=\sharp$ the problem is $(s,t)$-weakly tractable if and only if 
		\begin{equation*}
			s > \frac{1}{\alpha} \,\text{ and }\, t>0 
			\qquad \text{or} \qquad
			s > 0 \,\text{ and }\, t>1. 
		\end{equation*}
\end{itemize}
\end{theorem}

\begin{remark}
Let us add some comments on the previous result before we present its proof.
\begin{itemize}
	\item[(i)] First of all note that \autoref{thm:sobolev} allows to characterize weak tractability in all three cases. 
In particular, we closed the gap in \cite[Theorem~5.5]{KueSicUll14} for the case $\square=+$.
Furthermore, in \cite[Proposition~5.1]{KueSicUll14} it was shown that all these approximation problems are never quasi-polynomially tractable, i.e., it does not hold that there exist constants $C,\tau > 0$ such that
\begin{equation*}
	n^{\alpha,\square}(\epsilon,\id_d) \leq C \, \exp\!\left[ \tau \left( 1+ \ln \epsilon^{-1} \right)\left(1+ \ln d\right) \right], 
	\qquad \epsilon\in(0,1),\, d\in\N;
\end{equation*}
see, e.g., \cite{GneWoz11}. 
Since quasi-polynomial tractability obviously is a stronger notion than uniform weak tractability, we improved also this result.
\begin{corollary}\label{cor:sobolevWT}
	Let $\alpha>0$ be fixed. Then the approximation problem defined above is 
	\begin{equation*}
	\text{weakly tractable} \qquad \text{if and only if} \qquad
		\begin{cases}
			\alpha > 2 & \text{for} \quad \square = +,\\
			\alpha > 1 & \text{for} \quad \square = \sharp.
		\end{cases}
	\end{equation*}
	If $\square=*$, then we never have weak tractability.
	Moreover, for each $\square \in \{+,*,\sharp\}$ we never have uniform weak tractability.
\end{corollary}

\item[(ii)] Observe that (in contrast to \autoref{cor:sobolevWT}) \autoref{thm:sobolev} shows \emph{how much} the problem is getting easier with increasing smoothness $\alpha$, at least if $\square\in\{+,\sharp\}$. 

\item[(iii)] Furthermore, it is interesting to see that for this type of problems $\alpha$ only influences $s$ and not $t$. 
In contrast, the necessary conditions for $(s,t)$-weak tractability in \cite[Corollary~1]{Wei14} only depend on $t$ and not on $s$.

\item[(iv)] In sharp contrast to the results for linear tensor product problems---see \autoref{subsect:tensor}---in all three cases of \autoref{thm:sobolev} the obtained conditions show some kind of trade-off between the tractability parameters $s$ and $t$: 
one the one hand, independently of the smoothness~$\alpha$, we can achieve an arbitrarily good (subexponential) dependence of the information complexity $n^{\alpha,\square}(\epsilon,\id_d)$ on the accuracy $\epsilon$,  provided that we allow $t$ to be larger than $1$ which corresponds to a (super-)exponential dependence on the dimension $d$. On the other hand, if we seek for a moderate growth of $n^{\alpha,\square}(\epsilon,\id_d)$ with $d$, then $s$, i.e., the dependence on $\epsilon$, is restricted (by a term that involves $\alpha$). Anyhow, without sufficiently high smoothness, we cannot find bounds on the information complexity that show a nice dependence on $\epsilon$ and $d$ simultaneously.

\item[(v)] We finally remark that the result for $\square=*$ in \autoref{thm:sobolev} is sharp if we additionally assume $\alpha\geq 1/2$. For $0<\alpha<1/2$ there remains a gap. However, in all three cases we characterized $(s,t)$-weak tractability if we restrict ourselves to the interesting range of parameters $s,t\in(0,1]$.
\end{itemize}
\end{remark}

\begin{proof}[Proof of \autoref{thm:sobolev}]
\emph{Step 1 (Sufficient conditions for $\square=\sharp$).}
For $\alpha>0$ and $m\in\N$ let us define $E_m:=\left((m+1)^{-\alpha},m^{-\alpha}\right]$ such that $(0,1] = \bigcup_{m\in\N} E_m$. From \cite[Lemma~4.1]{KueSicUll14} we know that
\begin{equation*}
	a_{n,d}^{\alpha,\sharp} = (m+1)^{-\alpha}
	\quad \text{for all} \quad
	n\in\left( C(m-1,d), C(m,d) \right] 
	\quad \text{and} \quad 
	m\in\N,
\end{equation*}
where $C(m-1,d)\leq C(m,d)\leq 2^{\min\{d,m\}}\binom{m+d}{d}$, due to \cite[Lemma~3.4]{KueSicUll14}.
For all $\epsilon\in E_m$, $m\in\N$, and $d\in\N$ this implies 
\begin{align*}
	n^{\alpha,\sharp}(\epsilon,\id_d)
	\leq \inf\!\left\{n\in\N \sep a_{n+1,d}^{\alpha,\sharp} \leq (m+1)^{-\alpha} \right\}
	\leq C(m-1,d) 
	\leq \left[ 2\, (m+d) \right]^{\min\{d,m\}},
\end{align*}
i.e., due to the definition of $E_m$,
\begin{equation}\label{bound_n}
	\ln n^{\alpha,\sharp}(\epsilon,\id_d) 
	\leq \min\!\left\{\epsilon^{-1/\alpha},d\right\} \, \ln \left[ 2\left(\epsilon^{-1/\alpha}+d\right) \right]
	\quad \text{for all} \quad
	\epsilon \in(0,1], d\in\N.
\end{equation}

Now suppose that $s>1/\alpha$ and $t>0$. We like to adapt the technique given in \cite[Section~3.2.2]{Wei14a} and make use of Young's inequality. 
For this purpose, we choose $\delta\in(0,1)$ such that $p:=\alpha s(1-\delta)$ is strictly larger than $1$ and set $q:=p/(p-1)$, i.e., $1/p+1/q=1$.
From Young's inequality\footnote{Recall that Young's inequality states that we have $ab\leq a^p/p+b^q/q$ whenever $a,b\geq 0$ and $1/p+1/q=1$. We use this result for $a:=\epsilon^{-s(1-\delta)/p}$ and $b:=d^{t(1-\delta)/q}$.} it then follows that
\begin{align}
	\epsilon^{-s}+d^{\,t} 
	&= \left( \epsilon^{-s}+d^{\,t}\right)^{1-\delta} \left( \epsilon^{-s}+d^{\,t} \right)^\delta
	\geq c \, \left( \frac{\epsilon^{-s(1-\delta)}}{p} + \frac{d^{\,t(1-\delta)}}{q} \right) \left( \epsilon^{-s}+d^{\,t} \right)^\delta \nonumber\\
	&\geq c \, \epsilon^{-1/\alpha} \, d^{\,t(1-\delta)/q} \left( \epsilon^{-s}+d^{\,t} \right)^\delta \nonumber\\
	&\geq c' \, \epsilon^{-1/\alpha} \left(\epsilon^{-\min\{s,t\}\delta}+d^{\min\{s,t\}\delta} \right) \label{bound_eps}
\end{align}
with some $c'$ independent of $\epsilon$ and $d$.
Combining \link{bound_n} and \link{bound_eps} we conclude
\begin{align*}
	\frac{\ln n^{\alpha,\sharp}(\epsilon,\id_d)}{\epsilon^{-s}+d^{\,t}}
	\leq \frac{\min\!\left\{\epsilon^{-1/\alpha},d\right\} \, \ln \left[ 2\left(\epsilon^{-1/\alpha}+d\right) \right]}{c' \, \epsilon^{-1/\alpha} \left(\epsilon^{-\min\{s,t\}\delta}+d^{\min\{s,t\}\delta} \right)}
	\leq c'' \, \frac{\ln \left[2\left(\epsilon^{-\max\{1/\alpha,1\}}+d^{\max\{1/\alpha,1\}}\right) \right]}{\epsilon^{-\min\{s,t\}\delta}+d^{\min\{s,t\}\delta}}
\end{align*}
and finally, setting $x:=\delta \min\{s,t\}/\max\{1/\alpha,1\}>0$, we obtain
\begin{align*}
	\frac{\ln n^{\alpha,\sharp}(\epsilon,\id_d)}{\epsilon^{-s}+d^{\,t}}
	\leq \frac{c''}{x} \frac{\ln \left[2^x \left(\epsilon^{-\max\{1/\alpha,1\}}+d^{\max\{1/\alpha,1\}}\right)^x \right]}{\epsilon^{-\min\{s,t\}\delta}+d^{\min\{s,t\}\delta}}
	\leq c''' \, \frac{\ln \left[ 2^x \left(\epsilon^{-\min\{s,t\}\delta}+d^{\min\{s,t\}\delta}\right) \right]}{\epsilon^{-\min\{s,t\}\delta}+d^{\min\{s,t\}\delta}}
\end{align*}
which obviously tends to zero if $\epsilon^{-1}+d$ approaches infinity.
Hence, we have shown that $s>1/\alpha$ and $t>0$ implies $(s,t)$-weak tractability.

If we assume $t>1$ and $s>0$ (but not necessarily $s>1/\alpha$), we only need to exchange the roles of $\epsilon^{-1}$ and $d$. That is, we choose $\delta\in(0,1)$ such that $q:=t(1-\delta)>1$ and set $p:=q/(q-1)$. In this case the analogue of \link{bound_eps} reads
\begin{align*}
	\epsilon^{-s}+d^{\,t} 
	\geq c \, \epsilon^{-s(1-\delta)/p} \, d \left( \epsilon^{-s}+d^{\,t} \right)^\delta
	\geq c' \, d \left(\epsilon^{-\min\{s,t\}\delta}+d^{\min\{s,t\}\delta} \right).
\end{align*}
Therefore we conclude
\begin{align*}
	\frac{\ln n^{\alpha,\sharp}(\epsilon,\id_d)}{\epsilon^{-s}+d^{\,t}}
	\leq \frac{\min\!\left\{\epsilon^{-1/\alpha},d\right\} \, \ln \left[2\left(\epsilon^{-1/\alpha}+d\right) \right]}{c' \, d \left(\epsilon^{-\min\{s,t\}\delta}+d^{\min\{s,t\}\delta} \right)}
	\leq c'' \, \frac{\ln \left[2\left(\epsilon^{-\max\{1/\alpha,1\}}+d^{\max\{1/\alpha,1\}}\right) \right]}{\epsilon^{-\min\{s,t\}\delta}+d^{\min\{s,t\}\delta}}
\end{align*}
which converges to zero as before.

\emph{Step 2 (Sufficient conditions for $\square=+$).}
Due to \cite[Lemma~2.3]{KueSicUll14} we know that for fixed $\alpha>0$ the norm of the embedding $H^{\alpha,+} (\mathbb{T}^d) \hookrightarrow H^{\alpha/2,\sharp} (\mathbb{T}^d)$ is at most one, i.e., the unit ball in $H^{\alpha/2,\sharp} (\mathbb{T}^d)$ contains the unit ball w.r.t.\ the norm in $H^{\alpha,+} (\mathbb{T}^d)$.
Hence, we have $a_{n,d}^{\alpha,+} \leq a_{n,d}^{\alpha/2,\sharp}$ for all $n,d\in\N$, and, equivalently,
\begin{equation*}
	n^{\alpha,+}(\epsilon,\id_d) \leq n^{\alpha/2,\sharp}(\epsilon,\id_d)
	\quad \text{for all} \quad \epsilon\in(0,1] \quad \text{and} \quad d\in\N.
\end{equation*}
Using the previous step we therefore see that
\begin{equation*}
	\lim_{\epsilon^{-1}+d\nach\infty} \frac{\ln n^{\alpha,+}(\epsilon,\id_d)}{\epsilon^{-s}+d^{\,t}}
	\leq \lim_{\epsilon^{-1}+d\nach\infty} \frac{\ln n^{\alpha/2,\sharp}(\epsilon,\id_d)}{\epsilon^{-s}+d^{\,t}}
\end{equation*}
vanishes if $s > 2/\alpha$ and $t>0$, or if $s>0$ and $t>1$.

\emph{Step 3 (Sufficient conditions for $\square=*$).} 
Here we have to distinguish two cases. 
Let us first assume that $0 < \alpha \leq 1/2$. 
Then \cite[Lemma~2.3]{KueSicUll14} states that $H^{\alpha,*} (\mathbb{T}^d) \hookrightarrow H^{\alpha,\sharp} (\mathbb{T}^d)$ with norm one. 
Thus, the method presented in Step 2 together with the result from Step 1 implies $(s,t)$-weak tractability provided that $s>1/\alpha$ and $t>0$, or if $s>0$ and $t>1$. 

To handle the remaining cases where $\alpha>1/2$ we notice that the former result particularly shows 
that approximation in $H^{1/2,*} (\mathbb{T}^d)$ is $(s,t)$-weakly tractable if $s>2$ and $t>0$, or if $s>0$ and $t>1$.
Therefore, the claim immediately follows from the fact that for all $\alpha > 1/2$ the norm of the embedding $H^{\alpha,*} (\mathbb{T}^d) \hookrightarrow H^{1/2,*} (\mathbb{T}^d)$ is bounded by one; see \cite[Lemma~2.3]{KueSicUll14} again.

In conclusion the conditions $s>\max\{2,1/\alpha\}$ and $t>0$, or $s>0$ and $t>1$, are sufficient for $(s,t)$-weak tractability in the case $\square=*$.

\emph{Step 4 (Necessary conditions).}
In contrast to the authors of \cite{KueSicUll14} our (more general) necessary conditions are based on \emph{asymptotic} lower bounds for the respective approximation numbers. We give the proof for the case $\square = +$ in full detail:
In \cite[Theorem~4.15]{KueSicUll14} it has been shown that for all $d\in\N$, $n\geq 11^d\, e^{d/2}$, and every $\alpha>0$
\begin{equation*}
	a_{n,d}^{\alpha,+} 
	 \geq \left( \frac{1}{e\,(d+2)} \right)^{\alpha/2} \, n^{-\alpha/d}.
\end{equation*}
In turn, that means for all $d\geq 2$
\begin{equation*}
	a_{\ceil{11^d\, e^{d/2}}+1,d}^{\alpha,+} 
	 \geq \left( \frac{1}{e\,(d+2)} \right)^{\alpha/2} \, \left(\ceil{11^d\, e^{d/2}} +1 \right)^{-\alpha/d}
	 \geq \frac{1}{44^\alpha e^\alpha d^{\alpha/2}} =: 2^\alpha \epsilon_d > \epsilon_d
\end{equation*}
and thus $n^{\alpha,+}(\epsilon_d,\id_d)=\inf\!\left\{n\in\N \sep a_{n+1,d}^{\alpha,+} \leq \epsilon_d \right\} > \ceil{11^d\, e^{d/2}} > 11^d$. Therefore we have
\begin{equation*}
	\frac{\ln n^{\alpha,+}(\epsilon_d,\id_d)}{\epsilon_d^{-s}+d^{\,t}}
	> \frac{d \, \ln 11}{(88e)^{s\alpha} d^{s\alpha/2} +d^{\,t}} \geq \frac{C_{s,\alpha}}{d^{s\alpha/2-1}+d^{t-1}} \longrightarrow C \in (0, \infty], \quad d\nach\infty,
\end{equation*}
if $0\leq s\alpha/2 \leq 1$ and $0 \leq t \leq 1$ which shows that $s>2/\alpha$ and $t\geq 0$, or $s\geq 0$ and $t > 1$, is indeed necessary for $(s,t)$-weak tractability. 
For $\square\in\{*,\sharp\}$ we can argue similarly using the lower bounds given in \cite[Theorem~4.12 and Remark~4.13]{KueSicUll14}, as well as \cite[Theorem~4.6]{KueSicUll14}, respectively.

For each $\square \in \{+,*,\sharp\}$ it remains to check that $(s,t)$-weak tractability implies that $\min\{s,t\}$ is strictly positive. 
According to \autoref{rem:parameter} it suffices to find some $\epsilon_0^{\square} \in (0,1)$ for which $n^{\alpha,\square}(\epsilon_0^\square,\id_d) \nach \infty$, as $d$ grows to infinity, since this would contradict $(s,t)$-weak tractability with $s$ and/or $t$ being zero.
In the case $\square = *$ we use \cite[Lemma~4.8]{KueSicUll14} to conclude that
\begin{equation}\label{eq:a2d}
	a_{2d+1,d}^{\alpha,*} = 2^{-1/2} \qquad \text{for all} \quad d\in\N.
\end{equation}
Now, independently of $\alpha>0$, the choice $\epsilon_0^*:=2^{-1}$, say, does the job, because then $n^{\alpha,*}(\epsilon_0^*,\id_d)>2\,d$. 
Finally the cases $\square = +$ and $\square=\sharp$ can be deduced from \link{eq:a2d} as well since for every $n,d\in\N$ and all $\alpha>0$ we have
\begin{equation*}
	a_{n,d}^{\alpha,+} = ( a_{n,d}^{1,*} )^\alpha
	\qquad \text{and} \qquad
	a_{n,d}^{\alpha,\sharp} \geq a_{n,d}^{2\alpha,+} = ( a_{n,d}^{1,*} )^{2\alpha}
\end{equation*}
due to \cite[Formula (4.17)]{KueSicUll14} and Step 2 above. This completes the proof.
\end{proof}

\begin{remark}
We note in passing that the approximation numbers $a_{n,d}$ used in this section equal the square root of the singular values $\lambda_{d,n}$ discussed previously. 
Hence, instead of explicitly estimating $n(\epsilon,\id_d)/(\epsilon^{-s}+d^{\,t})$, we could have used \autoref{thm:general} in conjunction with the bounds proven in~\cite{KueSicUll14} to derive \autoref{thm:sobolev}.
But in any case more elaborate estimates in the pre-asymptotic regime (i.e., for small $n$) are needed to obtain sharp conditions for $(s,t)$-weak tractability for $\square=*$ and small $\alpha$.
\end{remark}

\autoref{thm:sobolev} provides a source of a variety of other tractability results related to approximation problems between Sobolev spaces.
To illustrate this point, let us recall the definition 
\begin{equation*}
	H^{\alpha,\beta,+}(\mathbb{T}^d) := \left\{ f\in L_2(\mathbb{T}^d) \sep \norm{f\sep H^{\alpha,\beta,+}(\mathbb{T}^d)} < \infty \right\}, \qquad \alpha,\beta\geq 0,
\end{equation*}
of (periodic) Sobolev spaces with so-called \emph{hybrid smoothness}, in which the norm is given by
\begin{equation*}
	\norm{f\sep H^{\alpha,\beta,+}(\mathbb{T}^d)}
	:= \left[ \sum_{k\in\Z^d} \abs{c_k(f)}^2 \left(1+\sum_{j=1}^d \abs{k_j}^2 \right)^\alpha \prod_{j=1}^d \left(1+\abs{k_j}^2 \right)^\beta \right]^{1/2}.
\end{equation*}
In addition, related spaces $H^{\alpha,\beta,\square}(\mathbb{T}^d)$ with $\square \in \{*,\sharp\}$ can be defined using straightforward modifications.

\begin{remark}
If $\beta=0$, then these spaces of with hybrid smoothness obviously coincide with $H^{\alpha,\square}(\mathbb{T}^d)$ defined above. On the other hand, setting $\alpha=0$, we obtain Sobolev spaces of dominating mixed smoothness $H^{\beta,\square}_{\mathrm{mix}}(\mathbb{T}^d)$ as considered, e.g., in \cite{KueSicUll13}.

More general, $H^{\alpha,\beta,\square}(\mathbb{T}^d)$ collects all periodic functions that possess a combination of isotropic smoothness of order $\alpha$ and dominating mixed regularity $\beta$. 
Spaces of this type have been introduced in \cite{GriKna2000}. They arise naturally from applications, e.g., in computational quantum chemistry \cite{Y10}. 
For details and further reading we refer to the recent preprints \cite{ByrDunSic+14, KaePotVol2013}.
\end{remark}

\begin{theorem}\label{thm:sobolev2}
Let $\alpha>0$. 
For $\beta,\gamma\geq 0$ and $\square\in\{+,*,\sharp\}$ consider the approximation problem
\begin{equation}\label{prob:id2}
	\widetilde{\id}_d \colon H^{\gamma+\alpha,\beta,\square} (\mathbb{T}^d) \nach H^{\gamma,\beta,\square}(\mathbb{T}^d), \quad d\in\N,
\end{equation}
w.r.t.\ the worst case setting and information from $\Lambda^{\mathrm{all}}$. 
Then all results from \autoref{thm:sobolev} and \autoref{cor:sobolevWT} transfer literally.
\end{theorem}

\begin{proof}
For $a,b\in\R$ and $\square=+$ we define the lifting operator $L_{a,b}^+$ by
\begin{equation}\label{def:lift}
	f \mapsto L_{a,b}^+ f 
	:= \sum_{\ell\in\Z^d} c_{\ell}(f) \left(1+\sum_{j=1}^d \abs{\ell_j}^2 \right)^a \prod_{j=1}^d \left(1+\abs{\ell_j}^2 \right)^b \exp\!\left( \i \, \ell\cdot \right)
\end{equation}
Then it is easily verified that $L_{a,b}^+$ is a linear isometric isomorphism between $H^{\gamma,\beta,+} (\mathbb{T}^d)$ and $H^{\gamma+a,\beta+b,+} (\mathbb{T}^d)$ whenever both spaces are well-defined. In addition, we obtain the factorization
\begin{equation}\label{eq:factor}
	\widetilde{\id}_d 
	= L_{\gamma,\beta}^+ \circ \id_d \circ L_{-\gamma,-\beta}^+,
\end{equation}
see \autoref{fig:factorization} below.
The multiplicativity of the approximation numbers thus implies
\begin{equation*}
	a_{n,d}^{\alpha,+}(\widetilde{\id}_d) 
	\leq \norm{L_{\gamma,\beta}^+} \cdot a_{n,d}^{\alpha,+}(\id_d) \cdot \norm{L_{-\gamma,-\beta}^+} 
	= a_{n,d}^{\alpha,+}(\id_d), \qquad n,d\in\N.
\end{equation*}
Using the fact that $(L_{a,b}^+)^{-1}=L_{-a,-b}^+$, the converse inequality is obtained analogously.
Consequently, from \link{eq:e_and_a} it follows that the $n$th minimal worst case errors (and hence the information complexities) of the approximation problems \link{prob:id2} and \link{prob:id1} coincide.
Thus, if $\square=+$, then the assertion is implied by \autoref{thm:sobolev} and \autoref{cor:sobolevWT}, respectively.
\begin{figure}[ht]
	\begin{center}
	\begin{tikzpicture}[descr/.style={fill=white}]
		\matrix (m) [matrix of math nodes, row sep=6em, column sep=6em, text height=3ex, text depth=1.2ex]{ 
		{ H^{\gamma+\alpha,\beta,+}\, } 	& {\, H^{\gamma,\beta,+} } \\
		{ H^{\alpha,0,+}=H^{\alpha,+}\, } 	& {\, L_2=H^{0,0,+} }\\ 
		};
\path[->, font=\small] (m-1-1) edge node[auto] {$\widetilde{\mathrm{id}}_d$} (m-1-2);
\path[->, dashed, font=\small] (m-1-1) edge node[auto, swap] {$L_{-\gamma,-\beta}^+$} (m-2-1.140);
\path[->, font=\small] (m-2-1) edge node[auto,swap] {$\mathrm{id}_d$} (m-2-2);
\path[->, dashed, font=\small] (m-2-2.50) edge node[auto, swap] {$L_{\gamma,\beta}^+$} (m-1-2);
\end{tikzpicture}
\end{center}
	\caption{Factorization described in Eq.~\link{eq:factor}.}\label{fig:factorization}
\end{figure}

In the remaining cases $\square\in \{*,\sharp\}$ we can argue similarly using straightforward modifications in the definition \link{def:lift}. 
Note that (due to the more complicated structure of the norm) for $\square=*$ the lifting operators will no longer be independent of $\gamma$ and $\beta$.
However, this does not harm our arguments.
\end{proof}

We conclude the discussion by some final remarks:
\begin{remark} 
Observe that \autoref{thm:sobolev2} covers \autoref{thm:sobolev} as special case in which $\gamma=\beta=0$. 
Moreover, the problems $H^{\gamma+\alpha,\square}\rightarrow H^{\gamma,\square}$ and $H^{\alpha,\beta,\square}\rightarrow H_{\mathrm{mix}}^{\beta,\square}$ with $\gamma,\beta\geq0$ and $\square\in\{+,*,\sharp\}$ are included as well.
In conclusion, the computational hardness of all these approximation problems solely depends on the difference $\alpha$ of the \emph{isotropic} smoothness in the source and the target space.

Combining the lifting argument used above with results proven in \cite{KueSicUll13} would allow to treat also the complementary situation in which the isotropic smoothness is kept fix and (a part of) the mixed regularity is approximated. 
Although problems of this type play an important role in practical applications, we do not discuss them here since they are known to be quasi-polynomially tractable; cf.\ \cite[Section~5.2]{KueSicUll13}. 
Thus, in this situation $(s,t)$-weak tractability holds for all $s,t>0$.
\end{remark}

\subsection{Integration of smooth functions}\label{sect:int}
Here we consider the multivariate integration problem
\begin{equation*}
	\mathrm{Int}_d \colon F_d \rightarrow \R, \quad f\mapsto \int_{[0,1]^d} f(x) \d x,\qquad d\in\N.
\end{equation*}
Therein the class of integrands
\begin{equation*}
	F_d := \left\{ f \in C^{\infty}([0,1]^d) \sep \norm{f \sep F_d}:= \sup_{k\in\N_0}\sup_{\theta \in \mathbb{S}^{d-1}} \norm{ D^k_\theta f \sep L_\infty([0,1]^d)} <\infty \right\}
\end{equation*}
is normed by the supremum over all directional derivatives $D_\theta^k f$ (of order $k$, in direction $\theta$) measured in $L_\infty([0,1]^d)$; see \cite{HinNovUll14} for details. 
As usual, we consider the worst case setting and linear algorithms that use at most $n\in\N_0$ function values (information from the class $\Lambda^{\mathrm{std}}$) to approximate $\mathrm{Int}_d$ on the unit ball $\B(F_d)$ of $F_d$. 

As the constants $\pm 1$ are contained in $\B(F_d)$ for all $d\in\N$, the initial error $\epsilon_d^{\mathrm{init}}$ for this problem is one. 
Therefore, we do not need to distinguish between the absolute and the normalized error criterion in what follows.
Again this justifies to write $n(\epsilon,\mathrm{Int}_d)$ instead of $n^{\mathrm{crit}}(\epsilon,\mathrm{Int}_d)$ for the corresponding information complexity.

To the best of our knowledge, by now the strongest result (from the information-based complexity point of view) known for the integration problem under consideration was given in \cite[Theorem~5]{HinNovUll14}. 
Let us restate it here for the reader's convenience:
\begin{lemma}\label{lem:cubature_rule}
	Let $d\in\N$. 
	Then for all $\epsilon\in (0,1]$ there exists an algorithm $Q(\epsilon,d)$ that uses at most
	\begin{equation*}
		N(Q(\epsilon,d)) 
		:= \exp\!\left\{ \ceil{\max\!\left\{4\sqrt{d},\, \ln \epsilon^{-1} \right\}} \cdot \left( 1+ \ln\!\left[ 1+\frac{d}{\ln \epsilon^{-1}}\right] \right) \right\}
	\end{equation*}
	function values to obtain a worst case error which is bounded by $\epsilon$.
\end{lemma}

Although $Q(\epsilon,d)$ from \autoref{lem:cubature_rule} is not well-suited for practical applications (see \cite{HinNovUll14} for details), it can be employed to derive the following tractablity assertion.

\begin{theorem}\label{thm:integration}
	The integration problem defined above is $(s,t)$-weakly tractable whenever $s>0$ and $t>1/2$.
\end{theorem}
\begin{proof}
Given $d\in\N$ we set $\epsilon_0(d):=\exp(-4\sqrt{d})$ and consider the algorithm
\begin{equation*}
	A(\epsilon,d)
	:= \begin{cases}
		Q(\epsilon_0(d),d) & \text{if} \quad \epsilon \in [\epsilon_0(d),1],\\
		Q(\epsilon,d) & \text{if} \quad \epsilon \in (0, \epsilon_0(d)).
	\end{cases}
\end{equation*}
Obviously, this modification of the cubature rule $Q$ from \autoref{lem:cubature_rule} still provides an $\epsilon$-approximation of $\mathrm{Int}_d$ on $\B(F_d)$.
If $\epsilon \geq \epsilon_0(d)$, then the number of function values used by $A(\epsilon,d)$ is upper bounded by
\begin{equation*}
	\exp\!\left\{ \ceil{4\sqrt{d}} \cdot \left( 1+ \ln\!\left[ 1+\frac{d}{\ln\epsilon_0(d)^{-1}}\right] \right) \right\}
	\leq \exp\!\left\{ 5\sqrt{d} \cdot \left( 1+ \ln\!\left[ 1+\sqrt{d} \right] \right) \right\}. 
\end{equation*}
For $\epsilon< \epsilon_0(d)$ we particularly have $d<\ln^2 \epsilon^{-1}$, so that in this case the corresponding bound reads
\begin{align*}
	\exp\!\left\{ \ceil{\ln \epsilon^{-1}} \cdot \left( 1+ \ln\!\left[ 1+\frac{d}{\ln \epsilon^{-1}}\right] \right) \right\} 
	&\leq \exp\!\left\{ \frac{5}{4} \ln \epsilon^{-1} \cdot \left( 1+ \ln\!\left[ 1+\ln \epsilon^{-1}\right] \right) \right\} \\
	&\leq \exp\!\left\{ \frac{5}{4} \ln \epsilon^{-1} \cdot \left( 1+ \ln \epsilon^{-1} \right) \right\}.
\end{align*}
In conclusion, for each $d\in\N$, the information complexity (w.r.t. the worst case setting) of the numerical integration problem under consideration satisfies
\begin{equation}\label{est:n}
	\ln n(\epsilon,\mathrm{Int}_d)
	\leq C \cdot \begin{cases}
		\sqrt{d}\cdot \ln d & \text{if} \quad \epsilon \in [\epsilon_0(d),1],\\
		\ln^2 \epsilon^{-1} & \text{if} \quad \epsilon \in (0, \epsilon_0(d)),
	\end{cases}
\end{equation}
where $C>1$ denotes some small universal constant.

Now let us fix some double sequence $((\epsilon_k,d_k))_{k\in\N}\subset(0,1)\times\N$ with $\epsilon_k^{-1}+d_k\rightarrow \infty$, as $k$ approaches infinity. 
We define the index set of all $k$ for which the second line in \link{est:n} applies by
\begin{equation*}
	I:= \{ k\in\N \sep \epsilon_k < \epsilon_0(d_k) \}.
\end{equation*}
Without loss of generality, we may assume that $\#I = \#(\N\setminus I)= \infty$. 
Suppose that there exists an infinite subset $J\subseteq I$ of indicee with $\inf\{\epsilon_k \sep k\in J \} = c > 0$. 
Then the definition of $\epsilon_0(d)$ would imply that 
the sequence $(d_k)_{k\in J}$ is uniformly upper bounded by some $d_0(c)\in\N$. This contradicts our assumption that $\epsilon_k^{-1}+d_k$ tends to infinity, as $k\rightarrow \infty$, so that we obtain
\begin{equation*}
	\lim_{I \ni k \rightarrow \infty} \epsilon_k = 0.
\end{equation*}
Moreover, a similar argument shows $d_k \rightarrow \infty$, as $k\in \N\setminus I$ approaches infinity.
Therefore, from \link{est:n} we conclude that 
\begin{equation*}
	\frac{\ln n(\epsilon_k,\mathrm{Int}_{d_k})}{\epsilon_k^{-s}+d_k^{\, t}} 
	\leq C \cdot \begin{cases}
		\displaystyle\frac{\ln d_k}{d_k^{\, t-1/2}} & \text{if} \quad k\in \N\setminus I, \vspace{4mm}\\
		\displaystyle\frac{\ln^2 \epsilon_k^{-1}}{\epsilon_k^{-s}} & \text{if} \quad k\in I
	\end{cases}
\end{equation*}
tends to zero as $k\rightarrow \infty$, provided that $s>0$ and $t>1/2$.
Since this holds for all double sequences $((\epsilon_k,d_k))_{k\in\N}\subset(0,1)\times\N$ with $\epsilon_k^{-1}+d_k\rightarrow \infty$, the proof is complete.
\end{proof}

\begin{remark}
Let us conclude this section with some final remarks:
\begin{itemize}
	\item[(i)] Using \cite[Theorem~4]{HinNovUll14} instead of \autoref{lem:cubature_rule} above would allow to prove a slightly worse complexity result, realized by an \emph{implementable} cubature rule. In fact, proceeding as before shows that the so-called \emph{Clenshaw-Curtis-Smolyak algorithm} is $(s,t)$-weakly tractable whenever $s>0$ and $t>2/3$. 
	
	\item[(ii)] \autoref{thm:integration} (as well as the preceding remark) significantly improves on the main assertions of \cite{HinNovUll14} in which classical weak tractability, i.e., $(s,t)$-weak tractability with $s=t=1$, was shown. 
	Moreover, in that paper it is stated that the proof of uniform weak tractability for the integration problem under consideration remains as an open problem. Although we also did not answer this question, our arguments indicate that a proof would require a new algorithm which uses much less integration nodes than $Q(\epsilon,d)$ from \autoref{lem:cubature_rule} if the accuracy $\epsilon$ is moderate compared to the dimension $d$, i.e., if $\epsilon \geq \epsilon_0(d)$.
	
	\item[(iii)] In \cite{HinNovUll+14} conditions for (uniform) weak tractability of a variety of related integration problems where investigated, but it is still not known whether integration on the probably most natural class
	\begin{equation*}
		\tilde{F}_d 
		:= \left\{ f \in C^{\infty}([0,1]^d) \sep \norm{f \sep \tilde{F}_d}:= \sup_{\alpha \in \N_0^d} \norm{ D^\alpha f \sep L_\infty([0,1]^d)} <\infty \right\}
	\end{equation*}
	of smooth functions (which is the slightly larger than $F_d$) suffers from curse of dimensionality, or not.
We conjecture that our notion of $(s,t)$-weak tractability can be used to obtain some new insights to these problems
which might help to finally answer this prominent question.
\end{itemize}
\end{remark}

\section*{Acknowledgement}
\addcontentsline{toc}{section}{Acknowledgments}
We appreciate comments from Henryk Wo{\'z}niakowski.

\addcontentsline{toc}{chapter}{References}
\bibliographystyle{is-abbrv}

\end{document}